\documentclass[11pt]{article}
\usepackage{amsfonts}
\usepackage{amssymb}
\usepackage{amsmath}
\usepackage{latexsym}
\usepackage{amsthm}
\usepackage{epsfig}
\usepackage{graphicx}
\usepackage{subfig}
\usepackage{color}
\usepackage{psfrag}

\newcommand{\PP}{L}
\newcommand{\QQ}{P}
\newcommand{\K}{Q}

\newcommand{\conv}{\mbox{conv}}
\newcommand{\cl}{\mbox{cl}}

\newcommand{\proj}{\mbox{proj}}

\newcommand{\aff}{\mbox{aff}}
\newcommand{\inte}{\mbox{int}}
\newcommand{\relint}{\mbox{relint}}
\newcommand{\rec}{\mbox{rec}}
\newcommand{\diam}{\mbox{diam}}

\def\st{\,|\,}

\newcommand{\old}[1]{{}}

\newtheorem{theorem}{Theorem}[section]
\newtheorem{corollary}[theorem]{Corollary}
\newtheorem{lemma}[theorem]{Lemma}

\newtheorem{observation}[theorem]{Observation}

\psfrag{L1}{$L^1$}
\psfrag{L2}{$L^2$}
\psfrag{L3}{$L^3$}
\psfrag{L4}{$L^4$}
\psfrag{L11}{$L^{1,1}$}
\psfrag{L12}{$L^{1,2}$}
\psfrag{L21}{$L^{2,1}$}
\psfrag{L22}{$L^{2,2}$}
\psfrag{L3b}{$\bar{L^3}$}

\psfrag{p1}{$p^1$}
\psfrag{p2}{$p^2$}

\psfrag{y1}{$y^1$}
\psfrag{y2}{$y^2$}
\psfrag{y3}{$y^3$}
\psfrag{y4}{$y^4$}

\psfrag{v1}{$v^1$}
\psfrag{v2}{$v^2$}
\psfrag{v3}{$v^3$}

\psfrag{M}{$M$}
\psfrag{M_0}{$M_0$}

\psfrag{H^1}{$H^1$}
\psfrag{H^2}{$H^2$}
\psfrag{H^3}{$H^3$}
\psfrag{H^4}{$H^4$}
\psfrag{H^F}{$H^F$}
\psfrag{H^*}{$H^*$}
\psfrag{H^A}{$H^A$}
\psfrag{H^B}{$H^B$}
\psfrag{barH^F}{$\bar H^F$}

\psfrag{h^F}{$h^F$}
\psfrag{h}{$h$}

\psfrag{barx}{$\bar x$}
\psfrag{x^H}{$x^H$}

\psfrag{Q_x}{$Q_x$}
\psfrag{Q^0_x}{$Q^0_x$}
\psfrag{Q^*_x}{$Q^*_x$}
\psfrag{Q^1_x}{$Q^1_x$}
\psfrag{Q^{sup}_x}{$Q^{sup}_x$}

\psfrag{D^1}{$D^1$}
\psfrag{D^2}{$D^2$}

\psfrag{d^1}{$d^1$}
\psfrag{d^2}{$d^2$}
\psfrag{dbarxH*}{$d(\bar x, H^*)$}
\psfrag{dbarxH*H1}{$d(\bar x, H^* \cap H^1)$}

\psfrag{A^{-1}}{$A^{-1}$}
\psfrag{A^0}{$A^0$}
\psfrag{A^1}{$A^1$}
\psfrag{A^2}{$A^2$}
\psfrag{A^{in}}{$A^{in}$}
\psfrag{A^{out}}{$A^{out}$}

\psfrag{B^{-1}}{$B^{-1}$}
\psfrag{B^0}{$B^0$}
\psfrag{B^1}{$B^1$}
\psfrag{B^2}{$B^2$}

\psfrag{C^{0,0}}{$C^{0,0}$}
\psfrag{C^{0,1}}{$C^{0,1}$}
\psfrag{C^{1,1}}{$C^{1,1}$}
\psfrag{C^{1,2}}{$C^{1,2}$}

\psfrag{F^{in}}{$F^{in}$}
\psfrag{F^{out}}{$F^{out}$}

\psfrag{w}{$w$}
\psfrag{diam(Q^*_x)}{$diam(Q^*_x)$}

\begin{document}

\title{Intersection Cuts with Infinite Split Rank}

\author{
Amitabh Basu${}^{1,2}$, \;
G\'erard Cornu\'ejols${}^{1,3,4}$\\
Fran\c{c}ois Margot${}^{1,5}$ \;
}

\date{April 2010}

\maketitle

\begin{abstract}
We consider mixed integer linear programs where free integer variables are
expressed in terms of nonnegative continuous variables.
When this model only has two integer variables,
Dey and Louveaux characterized the intersection cuts that have infinite
split rank. We show that, for any number of integer variables, the split
rank of an intersection cut generated from a bounded convex set $P$
is finite if and only if the integer points on the boundary of $P$ satisfy a certain
``2-hyperplane property''.
The Dey-Louveaux characterization is a consequence of this more general result.
\end{abstract}

\footnotetext[1] {Tepper School of Business, Carnegie Mellon University,
Pittsburgh, PA 15213.}

\footnotetext[2] {Supported by a Mellon Fellowship.}

\footnotetext[3] {LIF, Facult\'e des Sciences de Luminy,
Universit\'e de Marseille, France.}

\footnotetext[4] {Supported by  NSF grant CMMI0653419,
ONR grant N00014-09-1-0033 and ANR grant ANR06-BLAN-0375.}

\footnotetext[5] {Supported by  ONR grant N00014-09-1-0033.}
\section{Introduction.}
\label{SEC:Introduction}

In this paper, we consider mixed integer linear programs with
equality constraints expressing $m \geq 1$ free integer variables
in terms of $k \geq 1$ nonnegative continuous variables.
\begin{equation}
\label{SI}
\begin{array}{rrcl}
          & x & = & f + \displaystyle \sum_{j=1}^{k} r^j s_j \\[0.1cm]
           &  x  & \in & \mathbb{Z}^m \\
          &  s  & \in & \mathbb{R}_+^k .
\end{array}
\end{equation}

The convex hull $R$ of the solutions to \eqref{SI} is a corner polyhedron
(Gomory~\cite{gom69}, Gomory and Johnson~\cite{gj}). In the remainder we assume
$f \in \mathbb{Q}^m \setminus \mathbb{Z}^m$,
and $r^j \in \mathbb{Q}^m \setminus \left\{ 0 \right\}$. Hence $(x,s)=(f,0)$ is
not a solution of \eqref{SI}. To avoid discussing trivial cases, we
assume that $R \not= \emptyset$, and this implies that $\dim (R) =k$.
The facets of $R$ are the nonnegativity
constraints $s \geq 0$ and intersection cuts (Balas~\cite{bal}), 
namely inequalities
\begin{equation}
\label{EQ:Lineq}
\sum_{j=1}^k \psi(r^j) s_j \geq 1
\end{equation}
obtained from lattice-free convex sets $\PP \subset \mathbb{R}^m$ containing
$f$ in their interior, where $\psi$ denotes the gauge of $\PP-f$ 
(Borozan and Cornu\'ejols~\cite{bc}).
By {\em lattice-free} convex set, we mean a convex set with no point of
$\mathbb{Z}^m$ in its interior. By {\em gauge} of a convex set $\PP$ containing
the origin in its interior, we mean the function
$\gamma_{\PP} (r) = \inf \{ t > 0 \ | \ \frac{r}{t} \in \PP \}$.
Intersection cut~(\ref{EQ:Lineq}) for a given convex set $\PP$
is called {\em $\PP$-cut} for short.

When $m=2$,
Andersen, Louveaux, Weismantel and Wolsey~\cite{alww} showed that the only
intersection cuts needed arise from splits (Cook et al.~\cite{COOK}), 
triangles and quadrilaterals in the plane
and a complete characterization of the facet-defining inequalities
was obtained in Cornu\'ejols and Margot~\cite{cm}. More
 generally, Borozan and Cornu\'ejols~\cite{bc} showed that the only
intersection cuts needed in \eqref{SI} arise from full-dimensional maximal
lattice-free convex sets $\PP$. Lov\'asz~\cite{lovasz}
 showed that these sets are polyhedra with at most $2^m$ facets and that
they are
 {\em cylinders},  i.e., their recession  cone is a linear space. These
cylinders
${\PP} \subset \mathbb{R}^m$ can be written in the form
${\PP} = Q+V$ where $Q$ is a polytope of dimension at least one and $V$ is a
linear space such that $\dim (Q) + \dim (V) = m$. We say that $\PP$ is a
{\em cylinder over $Q$}. A {\em split} in $\mathbb{R}^m$ is a maximal
lattice-free cylinder over a line segment.

Let $\PP \subset \mathbb{R}^m$ be a polytope containing $f$ in its interior.
For $j = 1, \ldots, k$, let the {\em boundary point} for the ray $r^j$
be the intersection of the half-line
$\{ f + \lambda r^j \ | \  \lambda \geq 0 \}$ with the boundary of $\PP$.
We say that $\PP$ has {\em rays going into its corners} if each vertex
of $\PP$ is the boundary point for at least one of the rays $r^j$, $j= 1, \ldots, k$.

The notions of {\em split closure} and {\em split rank} were introduced
by Cook et al.~\cite{COOK} (precise definitions are
in Section~\ref{SEC:split}). They gave an example
 of \eqref{SI} with $m=2$ and $k=3$ that has an infinite split rank.
Specifically, there is a facet-defining inequality for $R$ that cannot
be deduced from a finite recursive application of the split closure
operation. This inequality is an intersection cut generated from a maximal
lattice-free triangle $L$ with integer vertices and rays going into its corners.
That triangle has vertices $(0,0)$, $(2,0)$ and $(0,2)$ and has
six integer points on its boundary. It is
a triangle of {\em Type 1} according to Dey and Wolsey~\cite{dw}.
Dey and Louveaux~\cite{DeyLou} showed that,
when $m=2$, an intersection cut has an infinite split rank if and only if
it is generated from a Type~1 triangle with rays going into its corners.
In this paper we prove a more general theorem, whose statement
relies on the following definitions.

A set $S$ of points in $\mathbb{R}^m$ is {\em 2-partitionable} if either
$|S| \le 1$ or there exists a partition of $S$ into {\em nonempty}
sets $S_1$ and $S_2$ and a split such that the points in $S_1$
are on one of its boundary hyperplanes and the points in $S_2$ are on the
other. We say that a polytope is
{\em 2-partitionable} if its integer points are 2-partitionable.

Let $\PP$ be a rational lattice-free polytope in $\mathbb{R}^m$ and let
$\PP_I$ be the convex hull of the integer points in $\PP$. We say that $\PP$ has
the {\em 2-hyperplane property} if
every face of $\PP_I$ that is not contained in a facet of $\PP$ is 2-partitionable.
Note that one of the faces of $\PP_I$ is $\PP_I$ itself and thus if $\PP$ has the
2-hyperplane property and $\PP_I$ is not contained in a facet of $\PP$, then there
exists a split containing all the integer points of $\PP$ on its boundary
hyperplanes, with at least one integer point of $\PP$ on each of the 
hyperplanes.

\begin{figure}[ht]
\centering \scalebox{0.8}{
\includegraphics{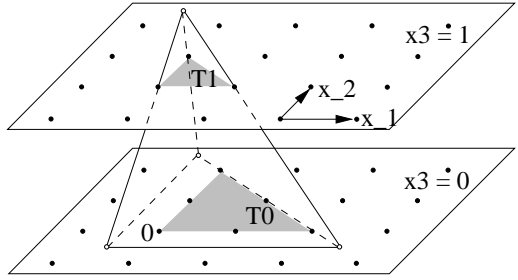}}
\caption{Illustration of the 2-hyperplane property.
}
\label{FIG:twohyp}
\end{figure}

We illustrate the 2-hyperplane property by giving an example in $\mathbb{R}^3$.
Consider the tetrahedron ${\PP} \subseteq \mathbb{R}^3$ given in Figure~\ref{FIG:twohyp}.
We will show that ${\PP}$ has the 2-hyperplane property.
Let $T_0$ be the shaded triangle with corners $(0, 0, 0), (2, 0, 0),$ and
$(0, 2, 0)$ and let $T_1$ be the shaded triangle with corners 
$(0, 0, 1), (1, 0, 1),$ and $(0, 1, 1)$. For any point $q \in \mathbb{R}^3$ 
with $1 < q_3 \le 2$, let  ${\PP}(q)$ be the tetrahedron obtained
as the intersection of the half-space $x_3 \ge 0$ with the cone with apex
$q$ and three extreme rays joining it to the corners of $T_1$.
Observe that for $t = (0, 0, 2)$, the vertices of ${\PP}(t)$ are $t$ and 
the corners of $T_0$. The tetrahedron ${\PP}$ depicted in
Figure~\ref{FIG:twohyp} is ${\PP}(p)$  for $p = (0.25, 0.25, 1.5)$. As $p$ is
in ${\PP}(t)$, and the intersection of ${\PP}(p)$ with the plane $x_3 = 1$ is 
the triangle $T_1$, it follows that the intersection of ${\PP}(p)$ with
the plane $x_3 = 0$ contains $T_0$. To check whether ${\PP}$ has the 
2-hyperplane property or not, we need to check if some of the faces of the 
convex hull ${\PP}_I$ of the integer points in ${\PP}$ are 2-partitionable or 
not. Note that ${\PP}_I$ is the convex hull of $T_1$ and $T_0$ and that a 
face of ${\PP}_I$ that is not contained in a facet of ${\PP}$ is either the 
triangle $T_1$ (which is 2-partitionable, using for example the split with 
boundary hyperplanes $x_1 = 0$ and $x_1 = 1$) or it contains an integer 
point in the plane $x_3 = 1$ and an integer point in the plane $x_3 = 0$ 
(and thus is 2-partitionable using these two planes as boundary for the 
split). As a result, ${\PP}$ has the 2-hyperplane property. 
We now give an example of a polytope ${\PP}'$ that does not have the 
2-hyperplane property. Define ${\PP}'(q)$ as ${\PP}(q)$ above, except that 
we change the half-space $x_3 \ge 0$ to $x_3 \ge -\frac{1}{2}$.
Let ${\PP}'$ be obtained by truncating ${\PP}'(t)$ by $x_3 \le \frac{3}{2}$, 
where $t = (0, 0, 2)$ as earlier. Then ${\PP}'_I$ is again the convex hull 
of $T_1$ and $T_0$. However $T_0$ is a face of ${\PP}'_I$ not contained in 
a facet of ${\PP}'$. Furthermore, $T_0$ is not 2-partitionable because it 
is not possible to find a split and a partition of the six integer points 
of $T_0$ into two nonempty sets $S_1, S_2$ with the property that 
$S_1$ lies on one boundary of the split and $S_2$ on the other. Therefore 
${\PP}'$ does not have the 2-hyperplane property.

The main result of this paper is the following theorem.

\begin{theorem}
\label{THM:GenDim}
Let $\PP$ be a rational lattice-free polytope in $\mathbb{R}^m$ containing $f$
in its interior and having rays going into its corners. The $\PP$-cut has finite split rank if and only if
$\PP$ has the 2-hyperplane property.
\end{theorem}

Given  a polytope $\PP \subset \mathbb{R}^m$ containing $f$ in its interior,
let $\PP_B$ be the convex hull of the boundary points for the rays
$r^1, \ldots, r^k$.
We assume here that nonnegative combinations of the rays $r^j$ in \eqref{SI} span $\mathbb{R}^m$.
This implies that $f$ is in the interior of $\PP_B$ and that $\PP_B$ has rays going into its
corners. Moreover, the $\PP$-cut and $\PP_B$-cut are identical.
Therefore the previous theorem implies the following.

\begin{corollary}
Assume that nonnegative combinations of the rays $r^j$ in \eqref{SI} span $\mathbb{R}^m$.
Let $\PP$ be a rational lattice-free polytope in $\mathbb{R}^m$ containing $f$
in its interior. The $\PP$-cut has finite split rank if and only if
$\PP_B$ has the 2-hyperplane property.
\end{corollary}

This corollary is a direct generalization of the characterization of
Dey and Louveaux for $m = 2$, as triangles of Type~1 do not have the
2-hyperplane property whereas all other lattice-free polytopes in the plane do,
as one can check using Lov\'asz'~\cite{lovasz} characterization of maximal
lattice-free convex sets in the plane.

The paper is organized as follows.
In Section~\ref{SEC:split}, we give a precise definition of split inequalities
and split rank, as well as useful related results.
In Section~\ref{SEC:space} we give an equivalent formulation of (\ref{SI}) 
that proves convenient to compute the split rank of $\PP$-cuts.
We prove one direction of Theorem~\ref{THM:GenDim} in Section~\ref{SEC:ratmax}
and we prove the other direction in Section~\ref{SEC:ratnonmax}.

\section{Split inequalities and split closure.}
\label{SEC:split}

Consider a mixed integer set 
$X = \{(x, y) \ | A x + B y \ge b, \ x \in \mathbb{Z}^p,
y \in \mathbb{R}^q\}$, where $A$ and $B$ are respectively
$m \times p$ and $m \times q$ rational matrices, and $b \in \mathbb{Q}^m$.
Let ${\K} = \{(x, y) \in \mathbb{R}^{p+q} \ | A x + B y \ge b\}$ be its linear 
relaxation.

Let $\pi \in \mathbb{Z}^p$ and $\pi_0 \in \mathbb{Z}$.
Note that all points in $X$ satisfy the {\em split disjunction induced} by
$(\pi, \pi_0)$, i.e.,
\begin{eqnarray*}
\pi x \le \pi_0 \hspace{1cm} \mbox{or} \hspace{1cm} \pi x \ge \pi_0+1 \ .
\end{eqnarray*}
The hyperplanes in $\mathbb{R}^{p+q}$ defined by
$\pi x = \pi_0$ and $\pi x = \pi_0+1$ are the
{\em boundary hyperplanes}
of the split. Conversely, two parallel hyperplanes
$H^1$ and $H^2$
with rational equations, both containing points $(x, y)$ with $x$ integer
and such that no points $(x, y)$ with $x$ integer are between them,
define a valid split $(\pi, \pi_0)$. Indeed, only $x$
variables can have nonzero coefficients in the equation of the planes and
we can assume that they are relatively prime integers. An application
of B\'ezout's Theorem~\cite{Seroul} shows that if, for $i = 1, 2$, the
equation of the hyperplane $H^i$ is given as $\pi x = h_i$, then
$|h_1 - h_2| = 1$. Define
\begin{eqnarray*}
&&{\K}^{^\le} = {\K} \cap \{(x, y) \in \mathbb{R}^{p+q} \ | \ \pi x \le \pi_0\} \ ,
\hspace{0.2cm}
{\K}^{^\ge} = {\K} \cap \{(x, y) \in \mathbb{R}^{p+q} \ | \ \pi x \ge \pi_0+1\} \\
&&{\K}(\pi, \pi_0) = \conv({\K}^{^\le} \ \cup \ {\K}^{^\ge}) \ .
\end{eqnarray*}
As $X \subseteq {\K}^{^\le} \ \cup \ {\K}^{^\ge}$, any inequality
that is valid for ${\K}(\pi, \pi_0)$
is valid for $\conv(X)$. The facets of ${\K}(\pi, \pi_0)$ that are not
valid for ${\K}$ are {\em split inequalities} obtained from the split
$(\pi, \pi_0)$. As shown by Cook et al.~\cite{COOK}, the intersection of all
${\K}(\pi, \pi_0)$ for all possible splits $(\pi, \pi_0)$ yields a polyhedron
called the {\em split closure} of ${\K}$.
Let the {\em rank-$0$ split closure} of ${\K}$ be ${\K}$ itself.
For $t = 1, 2, 3, \ldots$, the {\em rank-$t$ split closure} of ${\K}$
is obtained by taking the split closure of the rank-$(t-1)$ split closure
of ${\K}$.

Note that since ${\K}(\pi, \pi_0)$ is the convex hull of ${\K}^{^\le}$ with
${\K}^{^\ge}$, any point $(\bar x, \bar y)$ in ${\K}(\pi, \pi_0)$ is a
(possibly trivial)
convex combination of a point $p^1 \in {\K}^{^\le}$ and a point
$p^2 \in {\K}^{^\ge}$.
Moreover, if $(\bar x, \bar y)$ is neither in ${\K}^{^\le}$ nor in ${\K}^{^\ge}$, then
the segment $p^1p^2$ intersects $H^i$ in $q^i$ for $i = 1,2$. By convexity of
${\K}$, we have that $q^1 \in {\K} \cap H^1 = {\K}^{^\le} \cap H^1$ and
$q^2 \in {\K} \cap H^2 = {\K}^{^\ge} \cap H^2$. In summary, we have the following.

\begin{observation}
\label{OBS:ConvSplit}
\begin{itemize}
\item[(i)] If $(\bar x, \bar y) \in {\K}^{^\le}$ or $(\bar x, \bar y) \in {\K}^{^\ge}$,
then $(\bar x, \bar y) \in {\K}$ and $(\bar x, \bar y) \in {\K}(\pi, \pi_0)$;
\item[(ii)] If $(\bar x, \bar y) \in {\K}(\pi, \pi_0) \setminus 
({\K}^{^\le} \cup  {\K}^{^\ge})$, then
$(\bar x, \bar y)$ is a convex combination of
$(x^1, y^1) \in {\K} \cap H^1 = {\K}^{^\le} \cap H^1$
and $(x^2, y^2) \in {\K} \cap H^2 = {\K}^{\ge} \cap H^2$.
\end{itemize}
\end{observation}

Let $a x + b y \ge a_0$ be a valid inequality for $\conv(X)$ and let
$t$ be the smallest nonnegative integer such that the inequality is valid
for the rank-$t$ split closure of a polyhedron $K \supseteq X$ 
or $+\infty$ if no such integer exists.
The value $t$ is the
{\em split rank} of the inequality with respect to $K$.
It is known that valid inequalities for $\conv(X)$ may have infinite split
rank with respect to $K$ (Cook et al.~\cite{COOK}).

The following lemma gives three useful properties of split ranks
of inequalities.

\begin{lemma}
\label{LEM:rankface}
Let ${\K}$ be the linear relaxation of
$X = \{(x, y) \ | A x + B y \ge b, \ x \in \mathbb{Z}^p,
y \in \mathbb{R}^q\}$.
\begin{itemize}

\item[(i)] Let $X \subseteq {\K}_1 \subseteq {\K}$. The split rank with 
respect to ${\K}_1$ of a valid inequality for $\conv(X)$ is at most its 
split rank with respect to ${\K}$;

\item[(ii)]
Let $y'$ be a subset of the $y$ variables and let ${\K}(x,y')$
be the orthogonal projection of ${\K}$ onto the variables $(x, y')$.
Consider a valid inequality $\cal{I}$ for $\conv(X)$
whose coefficients for the $y$ variables not in $y'$
are all 0. The split rank of inequality $\cal{I}$ with respect to ${\K}(x,y')$ is greater than or equal to its split rank  with respect to ${\K}$.

\item[(iii)] Assume that all points in ${\K}$ satisfy an equality.
Adding any multiple of this equality to a valid inequality for $\conv(X)$
does not change the split rank of the inequality with respect to ${\K}$.
\end{itemize}

\end{lemma}

\begin{proof}
(i) Let $(\pi, \pi_0)$ be a split on the $x$ variables. We have
\begin{eqnarray*}
{\K}_1^{^\le} \subseteq {\K}^{^\le},
\hspace{1cm} \mbox{and}  \hspace{1cm}
{\K}_1^{^\ge} \subseteq {\K}^{^\ge} \ .
\end{eqnarray*}
It follows that the split closure of ${\K}_1$ is contained in the
split closure of ${\K}$ and that, for each $t = 0, 1, 2, \ldots$,
the rank-$t$ split closure of ${\K}_1$ is contained in the rank-$t$ split closure
of ${\K}$.

(ii) Let $\proj$ be the operation of projecting orthogonally onto the
variables $(x, y')$. It follows from the definitions of projection and convex hull that
the operations of taking the projection and taking the convex hull commute. Therefore we have,
for any split $(\pi, \pi_0)$ on the $x$ variables,
\begin{eqnarray*}
\proj\left(\conv({\K}^{^\le} \cup {\K}^{^\ge})\right) =
\conv\left(\proj({\K}^{^\le}) \cup \proj({\K}^{^\ge})\right) =
\conv\left({\K}(x,y')^{^\le} \cup {\K}(x,y')^{^\ge}\right) \ ,
\end{eqnarray*}
with the validity of the last equality coming from the fact that
none of the variables involved in the disjunction are projected out.
Hence, for all $t = 0, 1, 2, \ldots$, the projection $\proj$ of the
rank-$t$ split closure of ${\K}$ is contained in the rank-$t$ split closure of
${\K}(x,y')$ as the projection of an intersection of polyhedra is contained in the intersection of their projections.
The result then follows from the fact that inequality $\cal{I}$ is valid for
a polyhedron ${\K}'$ in the $(x, y)$-space if and only if it is
valid for $\proj({\K}')$.

(iii) For any $t= 0, 1, 2, \ldots$, all points in the rank-$t$ split closure of
${\K}$ satisfy the equality. An inequality is valid for the rank-$t$ split
closure of ${\K}$ if and only if the inequality obtained by adding to it
any multiple of the equality is.
\end{proof}

Let ${\K}_x \subseteq \mathbb{R}^p$ be a polyhedron 
where $\mathbb{R}^p$ is the space
of integer variables of the mixed integer set $X$. Assume that ${\K}_x$
is rational and full-dimensional. For each facet $F$ of ${\K}_x$
there exists a split $(\pi(F), \pi_0(F))$ with boundary hyperplanes
$H^1$ and $H^2$ parallel to $F$, with $F$ between $H^1$ and $H^2$
and with some points in ${\K}_x$ strictly between $F$ and $H^2$.
(Note that, if the hyperplane supporting $F$
contains an integer point, then $H^1$ supports $F$.)
Let the {\em width of the split} $(\pi(F), \pi_0(F))$ be the
Euclidean distance between $F$ and $H^2$.

As earlier, let ${\K} = \{(x, y) \in \mathbb{R}^{p+q} \ | A x + B y \ge b\}$.
 Performing a {\em round of splits around ${\K}_x$} on ${\K}$ means generating 
the intersection ${\K}'$ of ${\K}(\pi(F), \pi_0(F))$ for all facets $F$ of 
${\K}_x$. Note that if ${\K}$ contains the rank-$t$ split closure of
an arbitrary polytope ${\K}^*$, then  ${\K}'$ contains the rank-$(t+1)$ split
closure of ${\K}^*$.
Define the {\em width of a round of splits} around ${\K}_x$
as the minimum of the width of the splits $(\pi(F), \pi_0(F))$ for
all facets $F$ of ${\K}_x$.

\section{Changing space to compute the split rank.}
\label{SEC:space}

In the remainder of this paper, we will use $\QQ$ to denote 
the linear relaxation of \eqref{SI}.
Given a lattice-free polytope $\PP \subseteq \mathbb{R}^m$ containing $f$ in its 
interior, our goal is to compute the split rank of the $\PP$-cut with respect 
to $\QQ$. We show that this rank can be computed in another space that we
find convenient.

Let ${\QQ}^{\PP}(x, s, z)$ be the polyhedron obtained from ${\QQ}$
by adding one equation corresponding to the ${\PP}$-cut
(\ref{EQ:Lineq}) with a free continuous variable $z$ representing the
difference between its left and right-hand sides:
\begin{equation}
\label{SIz}
\begin{array}{rrcl}
          & x & = & f + \displaystyle \sum_{j=1}^{k} r^j s_j \\[0.1cm]
          & z & = & 1 - \displaystyle \sum_{j=1}^{k} \psi(r^j) s_j  \\[0.1cm]
          & x & \in & \mathbb{R}^m  \\
          & s & \in & \mathbb{R}_+^k \\
          & z & \in & \mathbb{R} \ .
\end{array}
\end{equation}

Clearly, $\QQ$ is the orthogonal projection of ${\QQ}^{\PP}(x, s, z)$ onto the 
$(x, s)$-space. As the relation between $\QQ$ and ${\QQ}^{\PP}(x,s,z)$ is a bijection projecting or adding a single continuous variable $z$, the split rank of 
(\ref{EQ:Lineq}) with respect to ${\QQ}$ or ${\QQ}^{\PP}(x, s, z)$ are identical.
Let ${\QQ}^{\PP}(x, z)$ be the orthogonal projection of ${\QQ}^{\PP}(x, s, z)$ 
onto the $(x, z)$-space. 
By Lemma~\ref{LEM:rankface} (iii), inequalities (\ref{EQ:Lineq}) and $z \le 0$ 
have the same split rank with respect to ${\QQ}^{\PP}(x, s, z)$.
By Lemma~\ref{LEM:rankface} (ii), the split rank of the inequality
$z \le 0$ for ${\QQ}^{\PP}(x, s, z)$ is smaller than or equal to its rank for ${\QQ}^{\PP}(x, z)$.
We thus have the following:

\begin{observation}
\label{OBS:rankQxz}
Let ${\PP} \subseteq \mathbb{R}^m$ be a lattice-free polytope containing $f$
in its interior and let ${\QQ}$ be the linear relaxation of \eqref{SI}.
The split rank of the ${\PP}$-cut (\ref{EQ:Lineq})
with respect to ${\QQ}$ is smaller than or equal to the split rank of the inequality
$z \le 0$ with respect to ${\QQ}^{\PP}(x, z)$.
\end{observation}

Let $0^k \in \mathbb{R}^k$ be the zero vector and let $e^j \in \mathbb{R}^k$ 
be the unit vector in direction $j$.
Observe that ${\QQ}^{\PP}(x, s, z)$ is a cone with apex $(f, 0^k, 1)$ and
extreme rays $ \{ (r^j, e^j,$ $-\psi(r^j)) \ | \ j = 1, 2, \ldots, k\}$.
As $(f, 1)$ is a vertex of ${\QQ}^{\PP}(x, z)$, the latter is
also a pointed cone. Its apex is $(f, 1)$ and its extreme rays are among
$\{(r^j, -\psi(r^j)) \ | \ j = 1, 2, \ldots, k\}$. 
Note also that if we embed ${\PP}$ in the hyperplane $z = 0$, then, for all
$j = 1, 2, \ldots, k$, the point
$p^j = f + \frac{1}{\psi(r^j)} r^j$ is on the boundary of ${\PP}$. The 
intersection of ${\QQ}^{\PP}(x, z)$ with the hyperplane $z = 0$ is the 
convex hull of the points $p^j$ for $j = 1, \ldots , k$.

Consider the pointed cone ${\QQ}^{\PP}$ with apex $(f, 1)$ and extreme rays 
joining $f$ to the vertices of ${\PP}$ embedded in the hyperplane $z  = 0$ 
(Figure~\ref{FIG:QPrays}).
Note that ${\QQ}^{\PP}(x, z)$ and ${\QQ}^{\PP}$ have the same apex and that all 
the extreme rays of ${\QQ}^{\PP}(x, z)$ are convex combinations of those of 
${\QQ}^{\PP}$. Thus, ${\QQ}^{\PP}(x, z) \subseteq {\QQ}^{\PP}$ and 
${\QQ}^{\PP}(x, z) = {\QQ}^{\PP}$ if and only if ${\PP}$ has rays going into 
its corners.

\begin{figure}[ht]
\centering \scalebox{0.8}{
\includegraphics{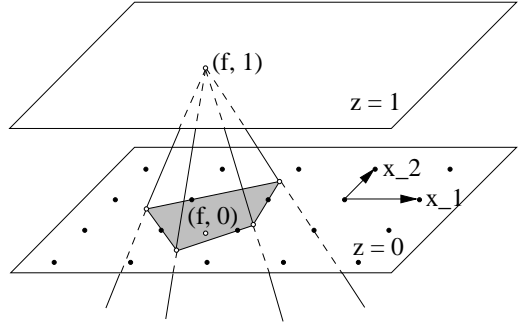}}
\caption{Illustration for the relation between ${\QQ}^{\PP}$,
${\QQ}^{\PP}(x, z)$ and  ${\PP}$. Polytope ${\PP}$
(shaded) is embedded in the plane $z = 0$ and contains $(f, 0)$;
cone ${\QQ}^{\PP}$ has apex $(f, 1)$ and
extreme rays joining $(f, 1)$ to the corners of ${\PP}$;
cone ${\QQ}^{\PP}(x, z)$ (not drawn) has apex $(f, 1)$ and
extreme rays joining $(f, 1)$ to points on the boundary of ${\PP}$.
}
\label{FIG:QPrays}
\end{figure}

Let $H$ be a polyhedron in the $(x, z)$-space.
For any $\bar x \in \mathbb{R}^m$, define the {\em height} of $\bar x$
with respect to $H$
as $\max\{\bar z \ |\ (\bar x, \bar z) \in  H\}$, with the convention that
this number may be $+\infty$ if $\{(\bar x, z) \in H\}$ is unbounded in the
direction of the $z$-unit vector or $-\infty$ if the maximum is taken
over an empty set. Define the {\em height} of $H$ as the maximum height of
$\bar x \in \mathbb{R}^m$ in $H$.
Observe that if the height of $H$ is $ < +\infty$ then the height with
respect to $H$ is a concave function over $\mathbb{R}^m$.
Note that the split rank of the inequality
$ z \le 0$ with respect to $H$ is at most $t$ if and only if  the rank-$t$ 
split closure of $H$ has height at most zero.

By Lemma~\ref{LEM:rankface} (i),
the split rank of $z \le 0$ with respect to ${\QQ}^{\PP}(x, z)$ is at most its split rank with
respect to ${\QQ}^{\PP}$.
With the help of Observation~\ref{OBS:rankQxz}, we get:

\begin{observation}
\label{OBS:rankQL}
Let ${\PP} \subseteq \mathbb{R}^m$ be a lattice-free polytope containing $f$
in its interior.
If the rank-$t$ split closure of ${\QQ}^{\PP}$ has height at most zero, then
the split rank of the ${\PP}$-cut (\ref{EQ:Lineq}) with respect to ${\QQ}$
is at most $t$. %Conversely, if ${\PP}$ has rays going into its corners and the split rank of the ${\PP}$-cut with respect to ${\QQ}$ is at most $t$, then the rank-$t$ split closure of ${\QQ}^{\PP}$ has height at most zero.
\end{observation}

\section{Proof of necessity.}
\label{SEC:ratmax}

In this section, we prove the ``only if'' part of Theorem~\ref{THM:GenDim}.
For a polyhedron ${\K}$, we denote by $\inte({\K})$ (resp. $\relint({\K})$) the
interior (resp. relative interior) of ${\K}$.

\begin{theorem}
\label{THM:Nec}
Let ${\PP}$ be a rational lattice-free polytope in $\mathbb{R}^m$ containing
$f$ in its interior and having rays going into its corners.
If the ${\PP}$-cut has finite split rank with respect to ${\QQ}$, 
then every face of ${\PP}_I$ that is
not contained in a facet of ${\PP}$ is 2-partitionable.
\end{theorem}

\begin{proof}
Assume that the ${\PP}$-cut has split rank $k$ for some finite $k$.
Thus, the height of the rank-$k$ split closure of
${\QQ}^{\PP}(x,s,z)$ is at most zero. By a theorem of Cook et al.~\cite{COOK},
the split closure of a polyhedron is a polyhedron. Therefore
there exists a finite number $t$ of
splits $(\pi^1, \pi^1_0), \ldots, (\pi^t, \pi^t_0)$ such that applying these
splits to ${\QQ}^{\PP}$ in that order reduces its height to zero or less.
Let ${\K}^0 := {\QQ}^{\PP}$ and ${\K}^j := {\K}^{j-1}(\pi^j, \pi^j_0)$ for 
$j = 1, 2, \ldots, t$.

Suppose for a contradiction that there exists a face $F$ of ${\PP}_I$
that is not 2-partitionable and
$F$ not contained in a facet of ${\PP}$.
As $F$ is not 2-partitionable, it must contain at least two integer points
and thus $\relint(F) \neq \emptyset$.
We claim that any point $\bar x \in \relint(F)$ has positive
height with respect to ${\K}^j$ for $j = 0, 1, \ldots, t$, a contradiction.

We prove the claim by induction on $j$.
For $j = 0$, the result follows from the fact that $\bar x \in \inte ({\PP})$
and every point in $\inte ({\PP})$ has a positive height with respect to ${\QQ}^{\PP}$.
Indeed, any point $\bar x \in \inte ({\PP})$ can be written as
$\bar x = \lambda f + (1-\lambda) x^*$ where $x^*$ is a point on the boundary
of ${\PP}$ with $0 < \lambda \le 1$. By convexity of ${\QQ}^{\PP}$, the height of $\bar x$
with respect to ${\QQ}^{\PP}$ is at least $\lambda \cdot 1 + (1-\lambda) \cdot 0 > 0$,
as the height of $f$ (resp. $x^*$) with respect to ${\QQ}^{\PP}$ is 1 (resp. 0).

Suppose now that $j > 0$ and that
the claim is true for $j-1$. If $F$ is contained in one of the boundary
hyperplanes of $(\pi^j, \pi^j_0)$ then $F$ is contained in ${\K}^{j-1^\le}$
or ${\K}^{j-1^\ge}$ and Observation~\ref{OBS:ConvSplit}~(i) shows that 
the height of any $\bar x \in \relint(F)$
with respect to ${\K}^j$ and ${\K}^{j-1}$ is identical, proving the claim for $j$.
Otherwise, as $F$ is not 2-partitionable, there exists an integer point $\bar p$
of $F$ that is not on the boundary hyperplanes of $(\pi^j, \pi^j_0)$.
Since $\bar p$ is integer, it is strictly on one of the two sides of the
split disjunction implying that there exists a point $x^* \in \relint(F)$
that is also strictly on one of the two sides of the split disjunction.
The height
of $x^*$ with respect to ${\K}^j$ and ${\K}^{j-1}$ is identical and positive
by induction hypothesis. All points on the relative
boundary of $F$ have non-negative height with respect to ${\K}^j$
as they are convex combinations of vertices of $F$ that have height zero,
and the height is a concave function.
As any point ${\bar x} \in \relint(F)$ is a convex combination of $x^*$
and a point on the boundary of $F$ with a positive coefficient for $x^*$,
the height of $\bar x$ with respect to ${\K}^j$ is positive.
This proves the claim.\end{proof}

\section{Proof of sufficiency.}
\label{SEC:ratnonmax}

Recall that $P$ denotes the linear relaxation of \eqref{SI}.
In this section, we prove the following theorem.

\begin{theorem}
\label{THM:main}
Let ${\PP}$ be a rational lattice-free polytope in $\mathbb{R}^m$ containing $f$
in its interior. If ${\PP}$ has the 2-hyperplane property,
then the ${\PP}$-cut has finite split rank with respect to ${\QQ}$.
\end{theorem}

Before giving the details of the proof, we present the main ideas. 
The first one is presented in Section~\ref{SEC:INTSPLIT}, where we
prove that Theorem~\ref{THM:main} holds when there is a
sequence of ``intersecting splits'' followed by an ``englobing split''.
These notions are defined as follows.

Let ${\K}$ be a polytope in $\mathbb{R}^m$ and let $(\pi, \pi_0)$ be a
split. The part of ${\K}$ contained between or on the boundary
hyperplanes $\pi x = \pi_0$ and $\pi x = \pi_0 + 1$ is denoted by
$\overline{{\K}(\pi, \pi_0)}$. The split $(\pi, \pi_0)$ is a
{\em ${\K}$-intersecting} split if both of its boundary hyperplanes
have a nonempty intersection with ${\K}$.
The split is {\em ${\K}$-englobing} if ${\K} = \overline{{\K}(\pi, \pi_0)}$.
Note that a split can be simultaneously englobing and intersecting, 
as illustrated in Figure~\ref{FIG:InterEnglob}.

\begin{figure}[ht]
\centering \scalebox{0.8}{
\includegraphics{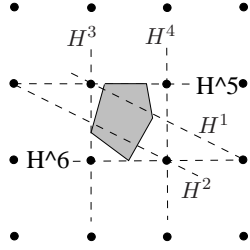}}
\caption{Illustration for intersecting and englobing splits. Polytope ${\K}$ is
shaded. Split with boundary hyperplanes $H^1$ and $H^2$ is ${\K}$-intersecting;
split with boundary hyperplanes $H^3$ and $H^4$ is ${\K}$-englobing;
split with boundary hyperplanes $H^5$ and $H^6$ is simultaneously
${\K}$-intersecting and ${\K}$-englobing.}
\label{FIG:InterEnglob}
\end{figure}

Let $(\pi^1, \pi^1_0), (\pi^2, \pi^2_0), \ldots, (\pi^t, \pi^t_0)$
be a finite sequence of splits. Recall the definition of ${\K}(\pi, \pi_0)$
from Section~\ref{SEC:split}. The polytopes
${\K}^0 := {\K}$, and ${\K}^j := {\K}^{j-1}(\pi^j, \pi_0^j)$ for $j = 1, \ldots,
t$ are the {\em polytopes obtained from the sequence} and ${\K}^t$ is
the {\em polytope at the end of the sequence}. When $S$ denotes the
sequence $(\pi^1, \pi^1_0), (\pi^2, \pi^2_0), \ldots, (\pi^t,
\pi^t_0)$ of splits, we will use the notation ${\K}(S)$ to denote the
polytope ${\K}^t$. We say that this sequence is a {\em sequence of
${\K}$-intersecting splits} if, for all $j = 1, \ldots, t$, we have
that $(\pi^j, \pi^j_0)$ is ${\K}^{j-1}$-intersecting.

Our approach to proving Theorem~\ref{THM:main} is to work with ${\QQ}^{\PP}$,
as introduced at the end of Section~\ref{SEC:split}, instead of with ${\QQ}$
directly. By Observation~\ref{OBS:rankQL}, to show that an ${\PP}$-cut has
finite split rank with respect to $\QQ$, it is sufficient to show that the 
height of ${\QQ}^{\PP}$ can
be reduced to at most 0 in a finite number of split operations.
Lemma~\ref{LEM:Dj} guarantees a reduction of the height of ${\QQ}^{\PP}$ when 
applying repeatedly a sequence of intersecting splits followed by an englobing
split. A key result to proving this
lemma is Lemma~\ref{LEM:ITERGEN2}, showing that, for an intersecting split,
we can essentially guarantee a constant reduction in height for the 
points cut off by the split. A consequence of these results is 
Corollary~\ref{COR:redHeight} that provides a sufficient condition for
proving that the $L$-cut obtained from a bounded rational lattice-free set $L$  
has finite split rank. It is indeed enough to exhibit a finite sequence of 
$L$-intersecting cuts followed by a final englobing cut.

Sections~\ref{SEC:ENLARGE}-\ref{SEC:MAIN}
deal with the unfortunate fact that it is not obvious that a
sequence of intersecting splits followed by an englobing split always exists.
However, since it is possible to reduce ${\PP}$ to ${\PP}_I$ using Chv\'atal 
cuts, the result is proved by replacing each of the Chv\'atal cuts by a finite
collection of intersecting splits for enlarged polytopes, and using
the 2-hyperplane property for proving that a final
englobing split exists. Section~\ref{SEC:ENLARGE}
shows how to enlarge polytopes so that they have desirable properties.
Section~\ref{SEC:CHVATAL} proves a technical lemma about Chv\'atal cuts.
In Section~\ref{SEC:MAIN}, induction
on the dimension is used to prove that the region removed by a
Chv\'atal cut can also be removed by a finite number of intersecting
splits for enlarged polytopes.

We now present details of the proof.

\subsection{Intersecting splits.}
\label{SEC:INTSPLIT}

The Euclidean distance between two points $x^1, x^2$ is denoted by
$d(x^1, x^2)$. For a point $x^1$ and a set $S$, we define $d(x^1, S)
= \inf\{d(x^1,y)\st y \in S\}$. The {\em diameter} of a polytope
${\K}$, denoted by $\diam({\K})$, is the maximum Euclidean distance between
two points in ${\K}$.

Let $Q_x$ be a rational polytope in $\mathbb{R}^m$ of dimension at least 1,
and let $M_0 < M$ be two finite numbers. Define $R(Q_x, M, M_0)$ (see
Figure~\ref{FIG:RMM0}) as the nonconvex region in the $(x, z)$-space
$\mathbb{R}^{m} \times \mathbb{R}$ containing
all points $(\bar x, \bar z)$ such that $\bar x$ is in the affine subspace
$\aff(Q_x)$ spanned by $Q_x$ and
$$
\bar z \le
\left\{
\begin{array}{ll}
M & \mbox{if \ } \bar x \in \relint(Q_x)\\
M_0 - \ \frac{d({\bar x}, Q_x)}{\diam(Q_x)} \ (M - M_0) & \mbox{otherwise.}
\end{array}
\right.
$$
Note that this definition implies that $\bar z \le M_0$ if $\bar x$ is
on the boundary of $Q_x$.

\begin{figure}[ht]
\centering \scalebox{0.8}{
\includegraphics{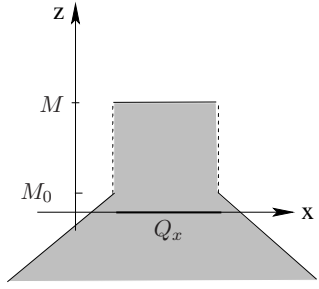}}
\caption{Illustration for the definition of $R(Q_x, M, M_0)$ (shaded area).}
\label{FIG:RMM0}
\end{figure}

The purpose of $R(Q_x, M, M_0)$ is to provide an upper bound on the height
of points outside $Q_x$ with respect to any polyhedron $Q$ of height $M$
having the property that the height with respect to $Q$ of any
$\bar x \not \in \relint(Q_x)$ is at most $M_0$, as shown in the
next lemma.

\begin{lemma}
\label{LEM:QinR}
Let $M_0 < M$ be two finite numbers, let $Q^0 \in \mathbb{R}^{m+1}$
be a rational polyhedron of height $M$ in the $(x, z)$-space, and
let $Q_x \in \mathbb{R}^m$ be a rational polytope
containing $\{x \ | \ \exists z > M_0 \mbox{ \ with \ } (x, z) \in Q^0\}$
in its relative interior. Then $Q^0 \subseteq R(Q_x, M, M_0)$.
\end{lemma}

\begin{proof}

Let $(\bar x, \bar z) \in Q^0$. As the height of $Q^0$ is $M$, we have
$\bar z \le M$. If $\bar x \in \relint(Q_x)$ then it follows that
$(\bar x, \bar z) \in R(Q_x, M, M_0)$. Otherwise $\bar x \not \in \relint(Q_x)$
and therefore $z \le M_0$ by definition of $Q_x$.
Let $x^M$ be a point in $Q_x$ with height $M$ with respect to $Q^0$.
Let $x^0$ be the intersection of the half-line starting at $x^M$ and going
through
$\bar x$ with the boundary of $Q_x$. Notice that this intersection
is on the segment $x^M{\bar x}$.
As the height of a point with respect to $Q^0$ is a concave function,
the height of $x^M$ is $M$, and the height
of $x^0$ is at most $M_0$ by choice of $Q_x$, we have that
$$
\bar z \le M_0 - \frac{(M-M_0)}{d(x^M, x^0)} \ d({\bar x}, x^0) \ .
$$
The result follows from $d(x^M, x^0) \le \diam(Q_x)$ and
$d({\bar x}, x^0) \ge d({\bar x}, Q_x)$.
\end{proof}

\begin{figure}[ht]
\centering \scalebox{0.8}{
\includegraphics{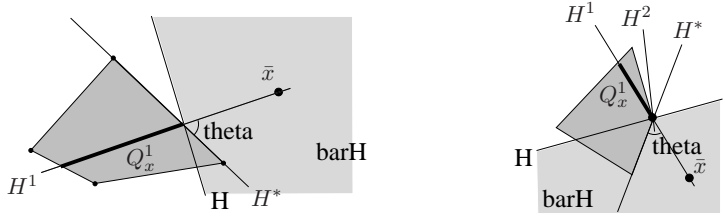}}
\caption{Illustration of the definition of $\theta(Q_x, H^1, \bar H)$ 
for $m = 2$. Polytope $Q_x$ is shaded, polytope $Q^1_x$ is the bold segment,
part of $\bar H$ is lightly shaded. On the left, rotating $H^1$ around
$H \cap H^1$ to get a hyperplane supporting a face of $Q_x$
yields a facet of $Q_x$ and $H^2 = H^*$. On the right, such a rotation 
could give the depicted hyperplane $H^2$, and an additional rotation 
gives $H^*$.}
\label{FIG:UBzGEN2}
\end{figure}

We define a $(Q_x, H^1, \bar H)$ {\em triplet} as follows:
$Q_x$ is a full-dimensional rational polytope in $\mathbb{R}^m$ with $m \ge 2$; 
$H^1 \subseteq \mathbb{R}^m$ is a hyperplane with 
$Q^1_x := H^1 \cap Q_x \neq \emptyset$;
$H \neq H^1$ is a hyperplane in $\mathbb{R}^m$ supporting a nonempty
face of $Q^1_x$ and $\bar H$ is a closed half-space bounded by $H$
not containing $Q^1_x$ (Figure~\ref{FIG:UBzGEN2}). Such a triplet
naturally occurs when using a $Q_x$-intersecting split $(\pi, \pi_0)$
when both $Q_x$ and $Q_x(\pi, \pi_0)$ are full dimensional: $H^1$ is one 
of the two boundary hyperplanes
of the split, $H$ is a hyperplane supporting a facet of $Q_x(\pi, \pi_0)$
that is not a facet of $Q_x$, and $\bar H$ is the half-space bounded
by $H$ that does not contain $Q_x(\pi, \pi_0)$.

Given a $(Q_x, H^1, \bar H)$ triplet,
we claim that there exists a hyperplane $H^*$ separating
$\inte(Q_x)$ from $\bar H \cap H^1$ with $H^* \cap H^1 = H \cap H^1$ and
maximizing the angle between $H^1$ and $H^*$. This maximum value is
denoted by $\theta(Q_x, H^1, \bar H)$.

To see that the claim holds, let $\bar x$ be a point in the relative 
interior of $\bar H \cap H^1$.
As $H \cap H^1$ is a hyperplane of $H^1$ separating $\bar x$ from $Q^1_x$,
we can rotate $H^1$ around $H \cap H^1$ to get a hyperplane $H^2$ supporting
a face of $Q_x$. Then, we can possibly rotate $H^2$ around $H \cap H^1$ 
to increase the angle between $H^2$ and $H^1$ while keeping the resulting
hyperplane $H^*$ supporting a face of $Q_x$. This rotation is stopped 
either if an angle of $\frac{\pi}{2}$ is obtained, or if $H^*$ gains a 
point of $Q_x$ outside of $H \cap H^1$.

\begin{lemma} \label{OBS:dist}
\begin{eqnarray}
d(\bar x, Q_x) \ge d(\bar x, H \cap H^1) \cdot \sin \theta(Q_x, H^1, \bar H) \ .
\end{eqnarray}
\end{lemma}

\begin{proof}
The observation follows from
\begin{eqnarray}
\label{EQ:Hstar}
d(\bar x, Q_x) \ge d(\bar x, H^*) =
d(\bar x, H^* \cap H^1) \cdot \sin \theta(Q_x, H^1, \bar H) =
d(\bar x, H \cap H^1) \cdot \sin \theta(Q_x, H^1, \bar H) \ .
\end{eqnarray}
Indeed, the first inequality comes from the fact that $H^*$ separates $\bar x$
from $Q_x$, the first equality is pictured in Figure~\ref{FIG:Hstar}, and the
last equality follows from $H^* \cap H^1 = H \cap H^1$.
\begin{figure}[ht]
\centering \scalebox{0.8}{
\includegraphics{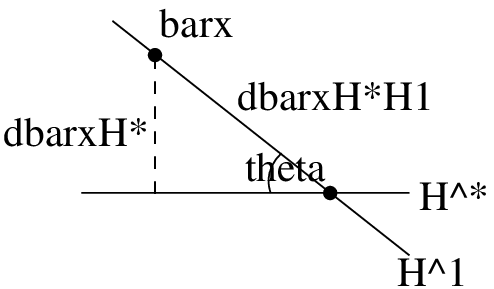}}
\caption{Illustration for the proof of
Lemma~\ref{OBS:dist}. 
%The symbol $\theta$ is a shorthand for $\theta(Q_x, H^1, \bar H)$.
}
\label{FIG:Hstar}
\end{figure}
\end{proof}

\begin{lemma}
\label{LEM:UBzGEN}
Consider a $(Q_x, H^1, \bar H)$ triplet.
Let $M^*_0 < M^*$ be two finite numbers and let $Q \subseteq R(Q_x, M^*, M^*_0)$ be a rational
polyhedron of height $M \le M^*$. Let $\bar x \in H^1 \cap \bar H$. The
height $\bar z$ of $\bar x$ with respect to $Q$ satisfies
$$
\bar z \le M^*_0 - \sin \theta(Q_x, H^1, \bar H)
\cdot \frac{(M^*-M^*_0)}{\diam(Q_x)} \cdot d(\bar x, H \cap H^1) \ .
$$
\end{lemma}

\begin{proof}
As $Q \subseteq R(Q_x, M^*, M^*_0)$ and $\bar x \not \in \relint(Q_x)$, we have that
\begin{eqnarray}
\label{EQ:UBzGEN1}
\bar z \le M^*_0 - \ \frac{d({\bar x}, Q_x)}{\diam(Q_x)} \ (M^* - M^*_0)\ .
\end{eqnarray}
If $\bar x$ is on the boundary of $Q_x$, then $\bar z \le M^*_0$,
 and the result holds since $d(\bar x, H \cap H^1) = 0$. Otherwise,
the result follows from using Lemma~\ref{OBS:dist}
to replace the term $d(\bar x, Q_x)$ in (\ref{EQ:UBzGEN1})  by
$d(\bar x, H \cap H^1) \cdot \sin \theta(Q_x, H^1, \bar H)$.
\end{proof}

\begin{figure}[ht]
\centering \scalebox{0.8}{
\includegraphics{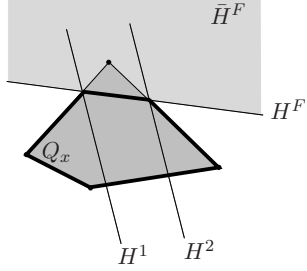}}
\caption{Illustration of the definition of
$\delta(Q_x, (\pi, \pi_0), {\bar H}^F)$ for $m = 2$. Polytope $Q_x$ is
shaded, part of $\bar H^F$ is lightly shaded, and $Q_x(\pi, \pi_0)$ is
depicted in bold line.}
\label{FIG:ITERGEN}
\end{figure}

Let $Q_x$ be a full-dimensional rational polytope in $\mathbb{R}^m$
with $m \ge 2$ and let $(\pi, \pi_0)$ be a $Q_x$-intersecting split
with boundary hyperplanes $H^1$ and $H^2$ (see Figure~\ref{FIG:ITERGEN}). Assume first that
$Q_x(\pi, \pi_0)$ is full-dimensional and strictly contained in $Q_x$,
and that the width $w$ of a round of splits around $Q_x$ satisfies
$w < \diam(Q_x)$. Recall that the width of a round of splits is defined at the
end of Section~\ref{SEC:split}. Let $H^F$ be a hyperplane supporting a facet
$F$ of $Q_x(\pi, \pi_0)$ that is not a facet of $Q_x$
and let ${\bar H}^F$ be the closed half-space not containing $Q_x(\pi, \pi_0)$
bounded by $H^F$. As mentioned earlier, we have that, for $i=1, 2$,
$(Q_x, H^i, {\bar H}^F)$ is a triplet. Let
\begin{eqnarray*}
\delta(Q_x, (\pi, \pi_0), {\bar H}^F) =
\frac{w}{\diam(Q_x)} \cdot \min \{\sin \theta(Q_x, H^1, {\bar H}^F), \
\sin \theta(Q_x, H^2, {\bar H}^F)\} \ .
\end{eqnarray*}
Define the {\em reduction coefficient} for $(Q_x, (\pi, \pi_0))$, denoted by
$\delta(Q_x, (\pi, \pi_0))$, as the minimum of
$\delta(Q_x, (\pi, \pi_0), {\bar H}^F)$ taken over all hyperplanes $H^F$
supporting a facet $F$ of $Q_x(\pi, \pi_0)$ that is not a facet of $Q_x$.
As $Q_x(\pi, \pi_0)$ has a finite number of facets, this minimum is
well-defined and its value is positive and at most one.
Assume now that $Q_x(\pi, \pi_0)$ is not full-dimensional or that
$Q_x = Q_x(\pi, \pi_0)$ or that $w \ge \diam(Q_x)$.
The reduction coefficient for $(Q_x, (\pi, \pi_0))$ is then defined as
the value 1.
Note that the reduction coefficient depends only on $Q_x$ and $(\pi, \pi_0)$
and always has a positive value smaller than or equal to 1.

Lemma~\ref{LEM:UBzGEN} can be used to prove a bound on the height of
some points $\bar x$ after applying an intersecting split. 
Given a set $S \subseteq \mathbb{R}^n$ we denote its closure by $\cl(S)$.
(We mean the topological closure here, not to be confused with the 
split closure.)

\begin{lemma}
\label{LEM:ITERGEN2}
Let $Q_x$ be a full-dimensional rational
polytope in $\mathbb{R}^m$ with $m \ge 2$. Let $(\pi, \pi_0)$ be a
$Q_x$-intersecting split and let $S$ be the sequence of a
round of splits around $Q_x$ followed by $(\pi,\pi_0)$.

Then, for any two finite numbers $M^*_0 < M^*$ and for any
rational polyhedron $Q \subseteq R(Q_x, M^*, M^*_0)$,
the height with respect to $Q(S)$ of any point in
$\cl(Q_x\setminus Q_x(\pi,\pi_0))$ is at most
$\max\left\{M_0^*, \ M - \ \delta(Q_x, (\pi, \pi_0)) \cdot
(M^*-M^*_0)\right\}$, where $M$ is the height of $Q$.
\end{lemma}

\begin{proof}
Let $w$ be the width of the round of splits around $Q_x$.
Observe that if $w \ge \diam(Q_x)$ then all the splits used during
the round of splits around $Q_x$ are $Q_x$-englobing. It follows
that the height of $Q(S)$ is at most $M_0^*$ and the result holds
(recall that if $Q(S) = \emptyset$ then its height is $-\infty$).
Similarly, if $Q_x(\pi, \pi_0)$ is not
full-dimensional, then $(\pi, \pi_0)$ is $Q_x$-englobing and the result
holds. Finally, if $Q_x(\pi, \pi_0) = Q_x$ then the result trivially
holds.

\begin{figure}[ht]
\centering \scalebox{0.8}{
\includegraphics{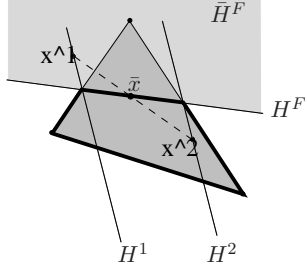}}
\caption{Illustration of the proof of Lemma~\ref{LEM:ITERGEN2}. Polytope 
$Q_x \subseteq \mathbb{R}^2$ is shaded, part of $\bar H^F$ is lightly shaded,
and $Q_x(\pi, \pi_0)$ is depicted in bold line. 
Point $\bar x$ is a convex combination of $x^1$ and $x^2$.}
\label{FIG:combconv}
\end{figure}

We can thus assume that $w < \diam(Q_x)$ and that $Q_x(\pi, \pi_0)$ is
full-dimensional and strictly contained in $Q_x$.
Let $H^1$ and $H^2$ be the boundary hyperplanes of
$(\pi, \pi_0)$. Let $\bar x \in \cl(Q_x\setminus Q_x(\pi,\pi_0))$
with maximum height $\bar z$ with respect to $Q(S)$.
If $\bar x$ is in a facet $F$ of $Q_x(\pi, \pi_0)$, let $H^F$ be a hyperplane
supporting $F$. Otherwise, let $H^F$ be a hyperplane
supporting a facet of $Q_x(\pi, \pi_0)$ separating $\bar x$ from
$Q_x(\pi, \pi_0)$. Let ${\bar H}^F$ be the half-space bounded by
$H^F$ not containing $Q_x(\pi, \pi_0)$.
We may assume $\bar z > M_0^*$ as otherwise the result trivially holds.
Thus $\bar x$ is in $\inte(Q_x)$ and strictly
between $H^1$ and $H^2$ (Figure~\ref{FIG:combconv}). Therefore, as shown in
Observation~\ref{OBS:ConvSplit} (ii), $(\bar x, \bar z)$ is a convex
combination of a point $(x^1, z^1) \in Q(S)$ with $x^1 \in H^1$ and
a point $(x^2, z^2) \in Q(S)$ with $x^2 \in H^2$, namely

\begin{eqnarray}
\label{EQ:combconv}
(\bar x, \bar z) = \frac{d(x^2, \bar x)}{d(x^1, x^2)} \cdot (x^1, z^1) +
\frac{d(x^1, \bar x)}{d(x^1, x^2)} \cdot (x^2, z^2) \ .
\end{eqnarray}

As all points in
$\bar H^F \cap H^i$ for $i = 1, 2$ are not in $\inte(Q_x)$ and thus
have height at most $M_0^*$, one of the points $x^i$
is in $Q_x \setminus {\bar H}^F$. Moreover, as $\bar x \in {\bar
H}^F$, the other one is in $\bar{H}^F$. Without loss of generality,
we assume that $x^2 \in Q_x \setminus {\bar H}^F$. For $i =
1, 2$, let $p^i$ be the closest point to $x^i$ in  $H^F \cap H^i$ and let
$d_i = d(x^i, p^i)$. We thus have $d_2 > 0$. As $z^1 \le M^*_0 < \bar z$,
we have $z^2 > z^1$ and points on the segment joining $(x^1, z^1)$ 
to $(x^2, z^2)$ have
increasing height as they get closer to $(x^2, z^2)$. We thus have that
$\bar x$ is on $H^F$, implying that $d_1 > 0$. Using (\ref{EQ:combconv})
and the similarity of triangles $x^1p^1\bar x$ and $x^2p^2\bar x$, implying
$\frac{d(x^i, \bar x)}{d(x^1, x^2)} = \frac{d^i}{d_1 + d_2}$, we get

\begin{eqnarray}
\label{EQ:combconvgen}
(\bar x, \bar z) = \frac{d_2}{d_1 + d_2} \cdot (x^1, z^1) +
\frac{d_1}{d_1 + d_2} \cdot (x^2, z^2) \ .
\end{eqnarray}

To simplify notation in the remainder of the proof, we use $\delta$
instead of $\delta(Q_x, (\pi, \pi_0))$.
By Lemma~\ref{LEM:UBzGEN} and the definition of $\delta$,
the height $z$ with respect to $Q$ of any
point $x \in {\bar H}^F \cap H^i$ is at most $M_0^* -\frac{\delta}{w} \cdot
(M^*-M^*_0) \cdot d(x, H^F \cap H^i)$.

It follows that we have
\begin{eqnarray}
\label{EQ:UB1}
z^1 \le M_0^* - \frac{\delta}{w} \cdot (M^*-M^*_0) \cdot d_1\ .
\end{eqnarray}

Using (\ref{EQ:UB1}) in (\ref{EQ:combconvgen}), we get
\begin{eqnarray}
\bar z &\le& \frac{d_2}{d_1 + d_2} \cdot \left(M_0^* - \frac{\delta}{w} \cdot
(M^*-M^*_0) \cdot d_1 \right) + \frac{d_1}{d_1 + d_2} \cdot z^2 \ .
\label{EQ:ubz}
\end{eqnarray}

Assume first that $d_2 \geq w$. As $z^2 \le M = M_0^* + (M-M_0^*)$,
(\ref{EQ:ubz}) becomes
\begin{eqnarray*}
\bar z &\le&  M_0^* +  \frac{d_1}{d_1 + d_2} \left((M-M_0^*) -
\delta \cdot \frac{d_2}{w} \cdot (M^*-M^*_0)  \right) \ .
\end{eqnarray*}
As we have that $ \bar z > M_0^*$, the expression in brackets above
is positive. Since the fraction in front of it is at most 1, we get
\begin{eqnarray*}
\bar z &\le& M_0^* + (M - M_0^*) - \delta \cdot \frac{d_2}{w}
(M^*-M^*_0) \le M - \delta \cdot (M^*-M^*_0) \ ,
\end{eqnarray*}
which proves the result.

Assume now that $d_2 < w$. We claim that
\begin{eqnarray}
\label{EQ:UB2}
z^2 \le M_0^* + \frac{d_2}{w} \cdot (M - M_0^*) \ .
\end{eqnarray}
To show this, we first claim that there exists a facet $F^1$ of $Q_x$ with
$d(x^2, F^1) \le d^2$. Recall that $p^2$ is the point in $H^F \cap H^2$ 
closest to $x^2$.
If $p^2$ is on the boundary of $Q_x$, then any facet $F^1$ of $Q_x$
containing $p^2$ proves the claim. Otherwise, as $H^F \cap H^2$ supports
$Q_x \cap H^2$, the segment $x^2p^2$ intersects
the boundary of $Q_x \cap H^2$ in a point that is on a facet
$F^1$ of $Q_x$, proving the claim.
Let $H(F^1)$ and $H'(F^1)$ be the boundary hyperplanes of the split
$(\pi(F^1), \pi_0(F^1))$ such that $H(F^1) \cap \inte(Q_x) =
\emptyset$ (Figure~\ref{FIG:combconvF1}).

\begin{figure}[ht]
\centering \scalebox{0.8}{
\includegraphics{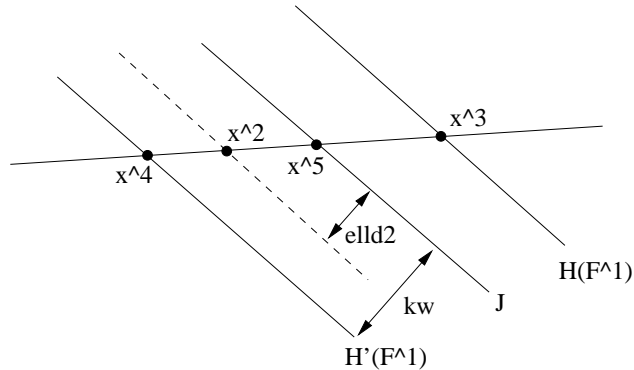}}
\caption{Illustration of the proof of Lemma~\ref{LEM:ITERGEN2}.}
\label{FIG:combconvF1}
\end{figure}

By definition of the width of a round of splits around
$Q_x$, the distance $k$ between $F^1$ and $H'(F^1)$ is at least $w$.
By Observation~\ref{OBS:ConvSplit} (ii), the point $(x^2, z^2)$ is a
convex combination of a point $(x^3, z^3) \in Q \cap H(F^1)$
and a point $(x^4, z^4) \in  Q \cap H'(F^1)$. The hyperplane $J$ 
in $\mathbb{R}^m$ supporting $F^1$ can be lifted in the $(x, z)$-space to 
the hyperplane $J' = \{(x, z) \in \mathbb{R}^{m+1} \ | \ x \in J\}$
orthogonal to the $x$-space.
Similarly, the hyperplane $H'(F^1) \in \mathbb{R}^m$ can be lifted
to a hyperplane $H''(F^1)$ in $\mathbb{R}^{m+1}$ orthogonal to the $x$-space.
The segment joining $(x^3, z^3)$ to $(x^4, z^4)$
intersects $J'$ in a point $(x^5, z^5) \in Q$ and thus

\begin{eqnarray}
\label{EQ:combconvF1}
(x^2, z^2) = \frac{d(x^2, x^5)}{d(x^4, x^5)} \cdot (x^4, z^4) +
\frac{d(x^2, x^4)}{d(x^4, x^5)} \cdot (x^5, z^5) \ .
\end{eqnarray}

Let $q'$ (resp. $q''$) be the closest point to $(x^2, z^2)$ in $J'$ 
(resp.  $H''(F^1)$) and let $\ell = d((x^2, z^2), q')$. Note that $\ell$
is at most the distance between $x^2$ and $F^1$ and recall that we 
have shown above that this is at most $d^2$.
Using similar triangles $(x^2, z^2)q'(z^4, z^4)$ and $(x^2, z^2)q''(z^5, z^5)$, 
we have
\begin{eqnarray}
\label{EQ:kell}
\frac{d(x^2, x^5)}{d(x^4, x^5)} = \frac{\ell}{k}
\hspace{1cm} \mbox{\ and \ } \hspace{1cm}
\frac{d(x^2, x^4)}{d(x^4, x^5)} = \frac{k - \ell}{k} \ .
\end{eqnarray}

Note that, trivially, $z^4 \le M$ and that $z^5 \le M^*_0$ as all points
$x$ on $J$ are either on the boundary or outside
of $Q_x$. Using these inequalities and (\ref{EQ:kell}) in
(\ref{EQ:combconvF1}), we get

\begin{eqnarray}
z^2 \le \frac{\ell}{k} \cdot M + \frac{k-\ell}{k} \cdot M^*_0 =
\frac{\ell}{k} \cdot (M - M^*_0) + M^*_0 \le
\frac{d^2}{w} \cdot (M - M^*_0) + M^*_0 \ ,
\end{eqnarray}

proving \eqref{EQ:UB2}. Using (\ref{EQ:UB2}) in
(\ref{EQ:ubz}), we obtain
\begin{eqnarray*}
\bar z &\le& M_0^* +  \frac{d_1}{d_1 + d_2} \cdot \frac{d_2}{w} \cdot
\left((M - M_0^*) - \delta \cdot (M^*-M^*_0)\right) \ .
\end{eqnarray*}
Since $ \bar z > M_0^*$, the expression in brackets above
is positive and each of the two fractions in front of it are positive
and at most 1. We thus get
\begin{eqnarray*}
\bar z &\le&  M - \delta \cdot (M^*-M^*_0) \ ,
\end{eqnarray*}
proving the result.
\end{proof}

 The next lemma plays
an important role in the proof of Theorem~\ref{THM:main}. Before stating the
lemma, we make the following observation, which will be used in its proof.

\begin{observation}
\label{LEM:closed}
Let $C_1$ be a full-dimensional convex set and let $C_2$ be
a closed set in $\mathbb{R}^m$. Then $C_1\setminus C_2$ is either
full-dimensional or empty.
\end{observation}

\begin{proof}
Suppose that $C_1\setminus C_2$ is not empty and let $x \in C_1\setminus
C_2$. Since $C_2$ is closed, there exists $\epsilon > 0$ such that
the closed ball $B(x,\epsilon) =\{y \in \mathbb{R}^m \st d(y, x)
\leq \epsilon\}$ does not intersect $C_2$. Therefore, $C_1 \cap
B(x,\epsilon) \subseteq C_1\setminus C_2$. But $C_1 \cap
B(x,\epsilon)$ is full-dimensional, as $x \in C_1$, $x\in
\inte(B(x,\epsilon))$, and $C_1$ is convex and
full-dimensional.
\end{proof}

The next lemma applies Lemma~\ref{LEM:ITERGEN2} iteratively $n$ times, with
a polytope ${\PP}^i \subseteq \mathbb{R}^m$ playing the role of $Q_x$ for
$i=1, \ldots, n$. For application $i$, this creates a region
$D^i = \cl({\PP}^i \setminus {\PP}^i(\pi^i, \pi^i_0))$ for which
a guaranteed height reduction is obtained. This guarantee is applied
to a polyhedron $Q \subseteq R({\PP}, M^*, M^*_0)$ where ${\PP}$ is a polytope
contained in ${\PP}^1$.
The statement of the lemma is illustrated in Figure~\ref{FIG:DJ}.

\begin{figure}[ht]
\centering \scalebox{0.8}{
\includegraphics{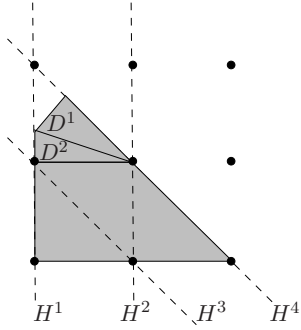}}
\caption{Illustration of the statement of Lemma~\ref{LEM:Dj} in a simple case.
Polytope ${\PP} = {\PP}^1$ is shaded,
${\PP}^2 = {\PP} \setminus D^1$, ${\PP}^3 = {\PP} \setminus (D^1 \cup D^2)$,
the boundary hyperplanes of $(\pi^1, \pi^1_0)$ are $H^1$ and $H^2$, and
the boundary hyperplanes of $(\pi^2, \pi^2_0)$ are $H^3$ and $H^4$.
The statement of the lemma is more general, as ${\PP}^i$ must merely contain
${\PP} \setminus (D^1 \cup \ldots \cup D^{i-1})$ instead of being equal to it, as
in this illustration.
}
\label{FIG:DJ}
\end{figure}

\begin{lemma}
\label{LEM:Dj}
Let ${\PP}$ be a full-dimensional rational polytope in $\mathbb{R}^m$ with
$m \ge 2$.
For $i = 1, \ldots n$, let ${\PP}^i$ be a rational
polytope in $\mathbb{R}^m$, let $(\pi^i, \pi^i_0)$ be
an ${\PP}^i$-intersecting split and
let $D^i = \cl({\PP}^i \setminus {\PP}^i(\pi^i, \pi^i_0))$ for $i = 1, \ldots
n$. Assume that ${\PP} \subseteq {\PP}^1$ and that
${\PP} \setminus (D^1 \cup \ldots \cup D^i)\subseteq {\PP}^{i+1}$, for $i =
1, \ldots, n-1$. Then there exists a finite sequence $S$ of splits
and a value $\Delta > 0$ such that, for any two finite numbers $M^*_0 < M^*$
and any rational polyhedron $Q \subseteq R({\PP}, M^*, M^*_0)$,
one of the following cases holds
\begin{itemize}
\item[(i)] the height of all points in $D^1$ is at most $M^*_0$ with respect to
$Q(S)$;
\item[(ii)] the height of all points in $D^1 \cup \ldots \cup D^n$ is at most
$M - \Delta \cdot (M^*-M^*_0)$ with respect to $Q(S)$, where $M$ is the 
height of $Q$.
\end{itemize}
\end{lemma}

\begin{proof}
Let $q$ be the smallest index $\{1, \ldots, n-1 \}$ such that
${\PP} \setminus (D^1 \cup \ldots \cup D^q)$ is not full-dimensional
and let $q = n$ if no such index exists. This ensures that ${\PP}^j$
is full-dimensional for $j = 1, \ldots, q$.
For $j = 1, \ldots, q$, let $R^j$ denote the round of splits around
${\PP}^{j}$. Let the sequence $S$ be $S = (R^1, (\pi^1, \pi^1_0), R^2,
(\pi^2, \pi^2_0), \ldots, R^q, (\pi^q, \pi^q_0))$.

Let $Q^1 = Q$ and let $Q^{j+1} = Q^{j}(R^j, (\pi^j, \pi^j_0))$,
and let $0 < \delta^j \le 1$ be the reduction coefficient for
$({\PP}^j, (\pi^j, \pi^j_0))$ for $j = 1, \ldots, q$.
Let $\Delta^1 = \delta^1$, let
$\Delta^j = \delta^j \cdot \Delta^{j-1}$ for $j = 2, \ldots, q$, and
observe that $\Delta^1 \ge \Delta^2 \ge \ldots \ge \Delta^q > 0$.
Let $\Delta := \Delta^q$.

Since $Q \subseteq R({\PP}, M^*, M^*_0)$ and ${\PP} \subseteq {\PP}^1$,
we trivially have $Q \subseteq R({\PP}^1, M^*, M^*_0)$. If
$M^*_0 \ge M - \Delta^1 \cdot (M^*-M^*_0)$, then Lemma~\ref{LEM:ITERGEN2} applied to
${\PP}^1$, $(\pi^1, \pi^1_0)$, and $Q$ proves that (i)
holds for $Q^2$, proving the result for $Q(S)$ as $Q(S) \subseteq Q^2$.
Therefore we may assume that $M^*_0 < M - \Delta^1 \cdot (M^*-M^*_0)$.

We claim that the height of any point in $D^1 \cup \ldots \cup D^j$
is at most $M - \Delta^j \cdot (M^*-M^*_0)$ with respect to 
$Q^{j+1}$ for $j = 1, \ldots, q$. We prove this claim by induction on $j$.

For $j = 1$, this is implied by Lemma~\ref{LEM:ITERGEN2} as it shows that
the height of all points in $D^1$ is at most $M - \Delta^1 \cdot (M^*-M^*_0)$ 
with respect to $Q^2$.

Assume now that the claim is true for some
$1 \le j < q$ and we prove it for $j + 1$. Let $M^{j+1}$
be the height of $Q^{j+1}$.
If $M^{j+1} \le M - \Delta^{j} \cdot (M^*-M^*_0)$, then the claim holds for $j+1$
as $\Delta^{j} \ge \Delta^{j+1}$.
Otherwise, as ${\PP} \setminus (D^1 \cup \ldots \cup D^{j})
\subseteq {\PP}^{j+1}$, the induction hypothesis implies that
 all points on the boundary or outside of ${\PP}^{j+1}$
have height at most  $M - \Delta^{j} \cdot (M^*-M^*_0)$ with respect to 
$Q^{j+1}$. Therefore Lemma~\ref{LEM:QinR} shows that
$Q^{j+1} \subseteq R({\PP}^{j+1}, M^{j+1}, M -\Delta^{j} \cdot (M^*-M^*_0))$.
We can thus use Lemma~\ref{LEM:ITERGEN2} to get that the
height of any point in $D^{j+1}$ with respect to $Q^{j+2}$ is at most
\begin{eqnarray}
\label{EQ:max}
\max\{M - \Delta^{j} \cdot (M^*-M^*_0), \ M^{j+1}- \delta^{j+1} \cdot (M^{j+1} -
(M - \Delta^{j} \cdot (M^*-M^*_0)))\} \ .
\end{eqnarray}
The second term can be rewritten as
\begin{eqnarray*}
&&(1 - \delta^{j+1}) \cdot M^{j+1} + \delta^{j+1} \cdot (M -
\Delta^{j} \cdot (M^*-M^*_0)) \\
&\le& (1 - \delta^{j+1}) \cdot M + \delta^{j+1} \cdot (M - \Delta^{j} \cdot (M^*-M^*_0))\\
&=&  M - \delta^{j+1} \cdot \Delta^{j} \cdot (M^*-M^*_0) .
\end{eqnarray*}

As $0 < \delta^{j+1} \le 1$, we obtain that the maximum in
(\ref{EQ:max}) is at most $M - \delta^{j+1}\cdot\Delta^{j} \cdot (M^*-M^*_0)$.
Thus, the height of any point in $D^1 \cup \ldots \cup D^{j+1}$ is
at most $M - \Delta^{j+1} \cdot (M^*-M^*_0)$ with respect to
$Q^{j+2}$, proving the claim for $j + 1$. This completes the proof of the claim.

If $q = n$, then point (ii) in the
statement of the lemma is satisfied, as we have shown that the
height of any point in $D^1 \cup \ldots \cup D^q$ is at most
$M - \Delta^q \cdot (M^*-M^*_0) = M - \Delta \cdot (M^*-M^*_0)$ with respect to
$Q^{q+1} = Q(S)$. If $q < n$ then, by definition of $q$, we have that
${\PP} \setminus (D^1 \cup \ldots \cup D^q)$ is not full-dimensional.
Observe that, as the $D^i$'s are closed sets,
$(D^1 \cup \ldots \cup D^q)$ is a closed set. Then
Observation~\ref{LEM:closed} implies that ${\PP} \setminus (D^1 \cup \ldots
\cup D^q)$ is empty. Since $M^*_0 \le M - \Delta^q \cdot (M^*-M^*_0)$ and 
since the height
of all points outside ${\PP}$ is at most $M^*_0$ with respect to $Q(S)$, 
this shows that the height of any point in the $x$-space is at most 
$M - \Delta^q \cdot (M^*-M^*_0)$ with respect to $Q(S)$, 
proving that (ii) is satisfied.
\end{proof}

This lemma is sufficient to prove the following corollary.

\begin{corollary}
\label{COR:redHeight}
Consider a mixed integer set of the form \eqref{SI} with $m = 2$. 
Let ${\PP} \subseteq \mathbb{R}^2$ be a rational lattice-free
polytope containing $f$ 
in its interior. Let $(\pi^1, \pi^1_0), \ldots, (\pi^n, \pi^n_0)$
be a sequence of ${\PP}$-intersecting splits and let ${\PP}^n$ be the polytope
obtained by applying the sequence on ${\PP}$. Let $(\pi, \pi_0)$
be an ${\PP}^n$-englobing split. Then there exists a finite number $q$
such that the height of the rank-$q$ split closure of ${\QQ}^{\PP}$ is 
at most zero.
\end{corollary}

\begin{proof} We prove the following more general result.
Let $Q_x$ be a full-dimensional rational polytope
in $\mathbb{R}^m$ with $m \ge 2$ and let $M^*_0 < M^*$ be two finite
numbers. Let $Q \subseteq R(Q_x, M^*, M^*_0)$ be a rational
polyhedron of height $M$ such that the height of any integer point in $Q_x$ is
at most $M^*_0$. Let $(\pi^1, \pi^1_0), \ldots, (\pi^n, \pi^n_0)$
be a sequence of $Q_x$-intersecting splits and let $Q_x^n$ be the polytope
obtained by applying the sequence on $Q_x$. Let $(\pi, \pi_0)$
be a $Q^n_x$-englobing split. Then there exists a finite number $q$
such that the height of the rank-$q$ split closure of $Q$ is at most $M^*_0$.

Let $Q^i_x$ for $i=1, \ldots, n$ be the polytopes obtained from
applying the sequence of splits on $Q_x$.
Let $S$ and $\Delta$ be obtained by applying Lemma~\ref{LEM:Dj}
for ${\PP} := Q_x$, ${\PP}^1 := Q_x$, ${\PP}^{i+1} := Q_x^{i}$ for
$i = 1, \ldots, n-1$, $M^*$, $M^*_0$ and $Q$. Assume that $M - M^*_0 \ge \Delta \cdot (M^*-M^*_0)$.
Then all points on the boundary of $Q_x^n$ have height at most
$M_n := \max \{M^*_0, \ M - \Delta  \cdot (M^*-M^*_0)   \}$ with
respect to  $Q(S)$.
Thus the height with respect to $Q(S)$ of any point on one of the boundary
hyperplanes of the disjunction $(\pi, \pi_0)$ is at most $M_n$, and
the height of $Q(S)(\pi, \pi_0)$ is at most $M_n$.
Applying the sequence $(S, (\pi, \pi_0))$ on $Q$ at most
$\lceil \frac{M-M^*_0}{\Delta \cdot (M^*-M^*_0)}\rceil$
times, we get a polyhedron $Q^1$ of height $M^1$ and for which all points in
$\cl(Q_x \setminus Q_x^1)$ have height at most $M^*_0$.
We can then iterate the above argument, applying Lemma~\ref{LEM:Dj}
in iteration $j = 1, \ldots, n-1$ to
${\PP} := Q^j_x$, ${\PP}^i := Q_x^{i+j-1}$ for $i=1, \ldots, n-j$,
$M^* := M^j$, $M^*_0$ and $Q := Q^j$. After these $n-1$ iterations, we get
a polyhedron $Q^n$ whose height is at most $M^*_0$.

The proof of the statement of the corollary follows by setting in the above
proof $Q_x := {\PP}$, $Q := Q^{\PP}$, $M^* := 1$ , and $M^*_0 := 0$.
\end{proof}

We note that this result can be used to prove the Dey-Louveaux theorem
stating that the split rank of any ${\PP}$-cut is finite whenever
${\PP} \subseteq \mathbb{R}^2$ is a maximal lattice-free rational polytope distinct
from a triangle of Type 1 with rays going into its corners.
Indeed, for each case (quadrilateral, triangles of
Type 2 or 3, and triangles of Type 1 with at least one corner ray missing),
one can exhibit explicitly the intersecting splits (at most two of them)
and the englobing split
required by Corollary~\ref{COR:redHeight}.
However, to handle the general case where ${\PP} \subset \mathbb{R}^m$ with
$m \ge 3$, a generalization of this corollary is presented in
Section~\ref{SEC:MAIN}.

\subsection{Enlarging the polyhedron.}
\label{SEC:ENLARGE}

The purpose of this section is to prove Theorem~\ref{LEM:enlarge} showing 
that it is possible 
to enlarge a lattice-free rational polytope ${\PP} \subset \mathbb{R}^m$ to a
rational lattice-free polytope ${\PP}'$ such that, for all facet $F$ of ${\PP}'$, 
the split $(\pi(F), \pi_0(F))$ is ${\PP}'$-intersecting and such that
${\PP}$ has the 2-hyperplane property if and only if ${\PP}'$ does. This is a useful
result, as this allows us to show that the effect of a Chv\'atal split
on a polytope can be obtained 
by a sequence of intersecting splits for the enlarged polytope. This
is developed in Section~\ref{SEC:CHVATAL}.

A {\em lattice subspace} of $\mathbb{R}^m$ is
an affine space $x+V$ where $x \in \mathbb{Z}^m$ and $V$ is a linear 
space generated 
by rational vectors. Equivalently, an affine space 
${\cal A} \subseteq \mathbb{R}^m$
is a lattice subspace if it is spanned by the integer points in $\cal A$ (see
Barvinok~\cite{barv} for instance). We need the following technical result.

\begin{lemma}
\label{LEM:rotation}
Let ${\PP}$ be a full-dimensional rational polytope in $\mathbb{R}^m$
with $m \ge 2$ given by $\{x \ | \ Ax \leq b\}$, where $A$ is
an integral matrix and $b$ is an integral vector. The rows of $A$
are denoted by $a_1,\ldots, a_n$. Suppose $a_1x \leq b_1$ defines a
facet of ${\PP}$ with no integer point contained in its affine hull $\{x
\ | \  a_1x = b_1\}$. Let $\tilde{A}$ be the matrix obtained from
$A$ by removing row $a_1$ and let $\tilde{b}$ be the vector obtained
from $b$ by removing its first component $b_1$. Then there exists a
rational inequality $a'x \leq b'$ such that $\tilde{{\PP}} := \{x \ | \
\tilde{A}x \leq \tilde{b}, a'x \leq b' \}$ contains the same set of
integer points as ${\PP}$, ${\PP} \subseteq \tilde{{\PP}}$, and the hyperplane
$\{x \ | \ a'x = b' \}$ contains integer points.
\end{lemma}

\begin{proof}
Let $G$ be the greatest common divisor of the coefficients in $a_1$.
As $\{x \ | \ a_1x = b_1\}$ does not contain any integer points,
$\frac{b_1}{G}$ is fractional by B\'ezout's Theorem~\cite{Seroul}. Consider the set ${\PP}' = \{x \ | \
\tilde{A}x \leq \tilde{b}, \; a_1x = G \cdot
\lceil\frac{b_1}{G}\rceil\}$. As $\rec({\PP}') \subseteq \rec({\PP})$ (where
$\rec({\PP})$ denotes the recession cone of ${\PP}$), ${\PP}'$ is a polytope
contained in the hyperplane $\{x \ | \ a_1x = G \cdot
\lceil\frac{b_1}{G}\rceil\}$. Since $\{x \ | \ a_1x = G \cdot
\lceil\frac{b_1}{G}\rceil\}$ is an $m-1$ dimensional lattice subspace,
its integer points can be partitioned into infinitely many parallel
lattice subspaces of dimension $m-2$. As ${\PP}'$ is bounded, one of them
is an $m-2$ dimensional lattice subspace $\cal A$ that does not
intersect ${\PP}'$. (See Figure~\ref{FIG:enlarge2}.)
This implies that we can choose a rational hyperplane 
$H \subseteq \mathbb{R}^m$ containing $\cal A$ and such that $H$
separates ${\PP}$ from ${\PP}'$ and does not contain a point in ${\PP}'$ 
(note that if ${\PP}'$ is empty, then we choose
$H$ containing $\cal A$ and such that ${\PP}$ is contained in one of the two
half-space bounded by $H$). Let $H = \{x \ | \ a'x = b'\}$ and such
that the half-space $a'x \leq b'$ contains ${\PP}$.

\begin{figure}[ht]
\centering \scalebox{0.8}{
\includegraphics{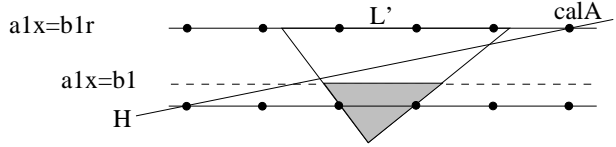}}
\caption{Illustration of the proof of Lemma~\ref{LEM:rotation}. Polytope
${\PP}$ is shaded, polytope ${\PP}'$ is a segment.
}
\label{FIG:enlarge2}
\end{figure}

By construction, the hyperplane $\{x \ | \ a'x = b'\}$ contains
integer points from $\cal A$.

Suppose for a contradiction that there exists an integer point
$p \in \tilde{{\PP}} \setminus {\PP}$.
Then $p$ satisfies $a_1x \geq G \cdot \lceil\frac{b_1}{G}\rceil$ and $a'p \leq b'$.
This implies that the segment joining $p$ to any point $\bar p$ in ${\PP}$
intersects the hyperplane $a_1x = G\lceil\frac{b_1}{G}\rceil$
at some point $p'$ (with, possibly, $p' = p$). Then $p' \in {\PP}'$
and, as all points in ${\PP}'$ violate the inequality $a'x \le b'$, we have
$p' \notin \tilde{{\PP}}$. This is a contradiction with the fact that
the segment $p \bar p$ is contained in $\tilde{{\PP}}$.
\end{proof}

We next make a couple of observations that are used in the remainder
of the paper. The proof of the first observation can be found in
Eisenbrand and Schulz~\cite{EisSchu}.

% Need to extend unimodular to unimodular + integer translation

\begin{observation}
\label{OBS:unimodular}
Let $\mathcal{M}$ be a unimodular
transformation of $\mathbb{R}^m$. Then integer points are mapped to
integer points and $(\pi, \pi_0)$ is a split if and only
if it is mapped to a split.
\end{observation}

\begin{observation}
\label{OBS:2-hyperplane}
Let ${\PP}$ and ${\PP}'$ be two polytopes having the same dimension,
such that ${\PP} \subseteq {\PP}'$ and ${\PP}_I = {\PP}'_I$.
If ${\PP}'$ has the 2-hyperplane property, then so does ${\PP}$.
\end{observation}

\begin{proof}
Consider any face $F$ of ${\PP}_I$ that is not contained in a facet of
${\PP}$. Face $F$ is not contained in a facet of
${\PP}'$, as the intersection of any facet of ${\PP}'$ with ${\PP}$ is
contained in a facet of ${\PP}$. Since ${\PP}'$ has the 2-hyperplane
property, $F$ is 2-partitionable.
\end{proof}

So far, we essentially dealt with full-dimensional lattice-free polytopes.
For the remainder of the paper, we need to extend a couple of definitions
to non-full dimensional polytopes. Let ${\PP}$ be a convex set in $\mathbb{R}^m$.
We say that ${\PP}$ is {\em lattice-free in its affine hull} if ${\PP}$ does
not contain an integer point in its relative interior. Note that when ${\PP}$
is full-dimensional, this definition is equivalent to ${\PP}$ being lattice-free.

Recall the definition of the  split $(\pi(F), \pi_0(F))$ given at the
beginning of Section~\ref{SEC:ratnonmax} for a facet $F$ of a
full-dimensional rational lattice-free polytope ${\PP} \subset \mathbb{R}^m$.
When ${\PP}$ is not full-dimensional, we consider the same definition
restricted to the affine hull $A$ of ${\PP}$. This split is uniquely
defined in $A$ and it will be denoted by $(\pi^A(F), \pi_0^A(F))$.
Note that the split $(\pi^A(F), \pi_0^A(F))$ can be extended
(in a non-unique way) to a split $(\pi(F), \pi_0(F))$ of $\mathbb{R}^m$.

\begin{theorem}
\label{LEM:enlarge}
Let ${\PP}$ be a rational polytope in $\mathbb{R}^m$ that is lattice-free in its
affine hull $A$. Assume that $\dim (A) \geq 1$ and that $A$
contains integer points.
Then it is possible to enlarge ${\PP}$ to a
rational lattice-free polytope ${\PP}' \subseteq A$ such that
${\PP}$ has the 2-hyperplane property if and only if ${\PP}'$ does, and such
that for each facet $F$ of ${\PP}'$
\begin{itemize}
\item[(i)] the affine hull of $F$ contains integer points;

\item[(ii)] the split $(\pi^A(F), \pi_0^A(F))$ is ${\PP}'$-intersecting.
\end{itemize}
\end{theorem}

\begin{proof}
The affine space $A$ is a lattice-subspace of dimension $t = \dim(A) \ge 1$
and therefore by choosing a lattice basis to define new coordinates, we may
assume that ${\PP}$ is full-dimensional in $\mathbb{R}^t$. Observe that
if $t = 1$, then ${\PP}$ always has the 2-hyperplane property and the result
holds when ${\PP}'$ is taken as the smallest segment with integer endpoints
containing ${\PP}$. Hence, in the remainder of the proof, we assume $t \geq 2$.
We obtain ${\PP}'$ in two phases.

{\bf Phase 1:} We first exhibit a rational lattice-free polytope
${\PP}^1$ containing ${\PP}$, such that ${\PP} \cap \mathbb{Z}^t = {\PP}^1 \cap
\mathbb{Z}^t $ and every ${\PP}^1$-englobing split is ${\PP}^1$-intersecting.

If ${\PP}$ is such that every ${\PP}$-englobing split is ${\PP}$-intersecting,
then ${\PP}^1$ is trivially taken to be ${\PP}$. Otherwise, consider an
${\PP}$-englobing split that is not ${\PP}$-intersecting and
let $H^1$ and $H^2$ be its two boundary hyperplanes. We assume
without loss of generality that $H^2 \cap {\PP} = \emptyset$. Let $C
\subseteq \mathbb{R}^t$ be the image of the unit hypercube $\bar C$
in $\mathbb{R}^t$ under a unimodular transformation $\mathcal{M}$,
such that all its vertices are in $H^1 \cup H^2$,
and such that the following condition (*) holds.

\begin{enumerate}
\item[(*)] If ${\PP} \cap H^1$ has dimension at least 1, then $C$
is chosen such that $C \cap {\PP} \cap H^1$ has dimension at least 1.
If ${\PP} \cap H^1$ is a single point, then $C$ is chosen such that $C
\cap {\PP} \cap H^1 = {\PP} \cap H^1$.
\end{enumerate}

Note that if  ${\PP} \cap H^1$ is empty, we are free to choose $\mathcal{M}$
as any unimodular transformation mapping the vertices of $\bar C$ to
integer points in $H^1 \cup H^2$.
We now apply the inverse of the unimodular transformation
$\mathcal{M}$ so that $C$ is transformed back into $\bar C$,
${\PP}$ is transformed into a polytope $Q$, and $H^i$ is transformed into
 hyperplane $\bar H^i$ for $i = 1,2$.
Without loss of generality, we can assume that $\bar H^1$ and $\bar
H^2$ are the hyperplanes $\{x \in \mathbb{R}^t \ \ | \ \ x_1 =0\}$ and
$\{x \in \mathbb{R}^t \ | \ x_1 = 1\}$. Observation~\ref{OBS:unimodular} shows
that condition (*) still holds for $\bar C$ and $Q$.
Let $\bar v$ be the center of the hypercube
$\bar C$, i.e. $\bar v = (\frac{1}{2}, \ldots, \frac{1}{2})$.
If $Q \cap \bar H^1$ has dimension at least 1, let $p^1$ be a rational fractional
point in $\bar C \cap Q \cap \bar H^1$.
If $Q \cap \bar H^1$ is of dimension less than 1, let $p^1$ be any rational fractional
point in $\bar C \cap \bar H^1$.  Let $q^1$ be the point in $\bar C \cap \bar H^2$
such that $\frac{p^1 + q^1}{2} = \bar v$.
Moreover we consider the set of the $2(t-1)$ centers of the facets of $\bar C$
other than the two supported by $\bar H^1$ and $\bar H^2$, i.e.,

\[
p^i_j = \left\{
\begin{array}{rl} \frac{1}{2} & j \neq i \\ 0 & j= i
\end{array}
\right.
\hspace{1cm}
q^i_j = \left\{ \begin{array}{rl} \frac{1}{2} & j \neq i \\
1 & j= i
\end{array} \right.
\hspace{1cm}
\mbox{for \ } i = 2, \ldots, t,  \ j = 1, \ldots, t.
\]

Let $S$ be the set of $2t$ points $p^i, q^i$ for $i = 1, \ldots, t$ and let
$Q_1 = \conv(S \cup Q)$. Clearly, $Q \subseteq Q_1$ and $Q_1$ is rational.
We now show that $Q \cap \mathbb{Z}^t = Q_1 \cap \mathbb{Z}^t $. Notice
that integer points in $Q_1$ are either on $\bar H^1$ or $\bar H^2$.
Since $Q_1 \cap \bar H^2$ is reduced to a single fractional point $q^1$,
$Q_1 \cap \bar H^2$ does not contain integer points.
When $Q \cap \bar H^1$ has dimension at least 1, our choice of $p^1$ implies
that $Q \cap \bar H^1 = Q_1 \cap \bar H^1$. When $Q \cap \bar H^1$ has dimension less than 1,
our choice of $p^1$ guarantees that the only integer point (if any) in
$Q_1 \cap \bar H^1$ is also in $Q \cap \bar H^1$.
Therefore, in both cases, $Q \cap \mathbb{Z}^t = Q_1 \cap \mathbb{Z}^t $.

We claim that the split closure of $Q_1$ contains the point $\bar v$ defined
above. Since $\bar v$ is an interior point of $Q_1$, this shows that
every $Q_1$-englobing split is $Q_1$-intersecting.

We prove the claim by showing that for any split ($\pi, \pi_0$), we have
$\bar v \in Q_1(\pi, \pi_0)$. This is obvious if $\pi \cdot \bar v \leq \pi_0$
or $\pi \cdot \bar v \geq \pi_0 +1$, so we may assume that $\pi_0 < \pi \cdot \bar v < \pi_0 +1$.
Since $\pi$ is integer and $\bar v$ has all coordinates
equal to $\frac{1}{2}$, we have that $\pi \cdot \bar v = \pi_0 + 0.5$. If
$\pi_j \neq 0$ for some $j \neq 1$, then
\[
\pi\cdot p^j = \pi\cdot \bar v - 0.5 \cdot \pi_j = \pi_0 + 0.5 - 0.5 \cdot \pi_j
\]
\[
\pi\cdot q^j = \pi\cdot \bar v + 0.5 \cdot \pi_j = \pi_0 + 0.5 + 0.5 \cdot \pi_j \ .
\]

The integrality of $\pi$ implies that $|\pi_j| \geq 1$, and thus
$p^j \in Q_1(\pi, \pi_0)$ and $q^j \in Q_1(\pi, \pi_0)$ and
$\bar v = \frac{p^j+q^j}{2} \in Q_1(\pi,\pi_0)$. On
the other hand, if $\pi_j = 0$ for all $j\neq 1$, then the
split disjunction must be $\{x_1 \leq 0 \vee x_1 \geq 1\}$ (since
$\pi_0 < \pi \cdot \bar v < \pi_0 +1$). But then,
$p^1 \in Q_1(\pi, \pi_0)$ and $q^1 \in Q_1(\pi, \pi_0)$ and
$\bar v = \frac{p^1+q^1}{2} \in Q_1(\pi,\pi_0)$. This completes the
proof of the claim.

Let ${\PP}^1 := \mathcal{M}(Q_1)$. Since $\mathcal{M}$
is a one-to-one map of splits into splits by Observation~\ref{OBS:unimodular},
any ${\PP}^1$-englobing split is ${\PP}^1$-intersecting.

If ${\PP} = {\PP}^1$ then, trivially, ${\PP}$ has the 2-hyperplane property if
and only if ${\PP}^1$ does. Otherwise, consider the ${\PP}$-englobing
split defined by $H^1$, $H^2$. Recall that this split is not
${\PP}$-intersecting, therefore ${\PP}_I$ is contained in $H^1$ and as
$H^1$ defines a face of ${\PP}$ and a face of ${\PP}^1$, both ${\PP}$ and ${\PP}^1$
have the 2-hyperplane property. Thus, ${\PP}$ has the $2$-hyperplane
property if and only if ${\PP}^1$ does.

{\bf Phase 2:} At the end of Phase 1, we obtain ${\PP}^1$ such that any
${\PP}^1$-englobing split is ${\PP}^1$-intersecting and ${\PP} \cap \mathbb{Z}^t
= {\PP}^1 \cap \mathbb{Z}^t$. Let ${\PP}' := {\PP}^1$ and apply the following
algorithm to ${\PP}'$.

\begin{enumerate}
\item If there exists a facet $F$ of ${\PP}'$ such that the affine hull
$A^F$ of $F$ does not contain integer points, do step 2 below.
Otherwise stop.

\item Using Lemma~\ref{LEM:rotation}, enlarge
${\PP}'$ to a polytope ${\PP}''$. Observe that the integer points in ${\PP}''$
are exactly those in ${\PP}'$ and that, compared to ${\PP}'$, ${\PP}''$ has
fewer facets whose affine hull does not contain integer points.
Rename ${\PP}' := {\PP}''$ and go to step 1.
\end{enumerate}

Since ${\PP}^1$ has a finite number of facets, the above algorithm
terminates. The polytope at the end of the algorithm satisfies the
statement of the lemma. Indeed, first notice that ${\PP}^1 \subseteq
{\PP}'$, where ${\PP}^1$ is the polytope obtained at the end of Phase 1.
Second, by construction, every facet $F$ of ${\PP}'$ satisfies
condition (i) in the statement of the lemma. Moreover, if
$(\pi(F), \pi_0(F))$ is not ${\PP}'$-intersecting, then it is 
${\PP}'$-englobing. As ${\PP}^1 \subseteq {\PP}'$, this split is also 
${\PP}^1$-englobing. By construction of ${\PP}^1$, it is 
${\PP}^1$-intersecting and thus also ${\PP}'$-intersecting,
a contradiction. Therefore (ii) holds.

Finally, we show that ${\PP}^1$ satisfies the 2-hyperplane property if
and only if ${\PP}'$ does. As ${\PP}^1 \subseteq {\PP}'$, ${\PP}^1_I = {\PP}'_I$ and $\dim({\PP}^1) = \dim({\PP}')$,
Observation~\ref{OBS:2-hyperplane} shows one of the two
implications. For the converse, suppose that ${\PP}^1$ has the
2-hyperplane property. Observe that any facet of ${\PP}^1$ containing
integer points is contained in a facet of ${\PP}'$. This shows that if $F$ is a
face of ${\PP}'_I$ that is not contained in a facet of ${\PP}'$, then $F$ is
not contained in a facet of ${\PP}^1$ and is thus 2-partitionable. It
follows that ${\PP}'$ has the 2-hyperplane property.
 \end{proof}

\subsection{Chv\'atal splits.}
\label{SEC:CHVATAL}

Given a polytope $Q$ in $\mathbb{R}^m$, a {\em Chv\'atal split}
is a split $(\pi, \pi_0)$ such that
$Q \cap \{ x \in \mathbb{R}^m \ | \ \pi x \geq \pi_0+1 \} = \emptyset$.
Note that $Q_I \subset Q \cap \{ x \in \mathbb{R}^m \ | \ \pi x \leq \pi_0 \}$.

The goal of this section is to prove a technical lemma about Chv\'atal splits.
The statement of the lemma is illustrated in Figure~\ref{FIG:swiping}.

\begin{figure}[ht]
\centering \scalebox{0.8}{
\includegraphics{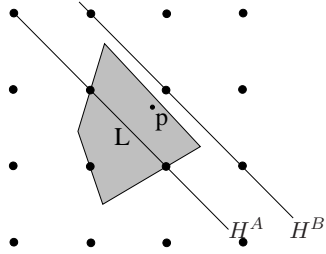}}
\caption{Illustration for the statement of Lemma~\ref{LEM:swiping} with
$m = 2$. Polytope $Q$ is shaded and polytope ${\PP}$ is a segment.}
\label{FIG:swiping}
\end{figure}

\begin{lemma}
\label{LEM:swiping}
Let $Q$ be a full-dimensional rational
lattice-free polytope in $\mathbb{R}^m$ with $m \ge 2$. Let $(\pi,
\pi_0)$ be a Chv\'atal split and let $H^A := \{ x\ | \  \pi x = \pi_0
\}$ and $H^B := \{ x\ | \  \pi x = \pi_0+1 \}$ be its boundary
hyperplanes, with $H^B \cap Q$ empty. Assume that ${\PP} := H^A \cap Q$
has dimension $\dim({\PP}) = m-1$, and for each facet $F$ of ${\PP}$ the
affine hull of $F$ contains integer points and the split
$(\pi^{H^A}(F), \pi_0^{H^A}(F))$ is ${\PP}$-intersecting. Let $p$ be any
point strictly between $H^A$ and $H^B$. Then there exists a finite
sequence $S$ of ${\PP}$-intersecting splits such that $Q(S) \cap \{ x
\in \mathbb{R}^m\ | \ \pi x \geq \pi_0 \}$ is contained in the
pyramid $\conv(p \cup {\PP})$.
\end{lemma}

\begin{proof}
Let $F$ be a facet of ${\PP}$. By hypothesis, there exists a split
$(\pi^F, \pi^F_0)$ with boundary hyperplanes $H^{F}_0$ and $H^{F}_1$
such that $H^{F}_0 \cap H^A$ contains $F$ and $H^{F}_1 \cap {\PP} \neq
\emptyset$. All the integer points in $\mathbb{R}^m$ can be
partitioned on equally spaced hyperplanes parallel to $H^A$, and the
integer points in any of these hyperplanes can themselves be
partitioned on equally spaced affine subspaces parallel to $H^{F}_0
\cap H^A$. Let $A^k$ with $k \in \mathbb{Z}$ denote these affine
subspaces in $H^A$, where $A^0 := H^{F}_0 \cap H^A$, $A^1 := H^{F}_1
\cap H^A$, and  $A^j$ is between $A^i$ and $A^k$ if and only if $i <
j < k$. Let $B^0$ be an affine subspace in $H^B$ parallel to $A^0$.
Let $a_0 \in A^0$, $b_0 \in B^0$ and let $v$ be the vector $b_0 - a_0$. Define
$B^k = A^k + v$ with $k \in \mathbb{Z}$ to be the translate of $A^k$.
 (See an illustration in Figure~\ref{FIG:swipingpf}.)

\begin{figure}[ht]
\centering \scalebox{0.8}{
\includegraphics{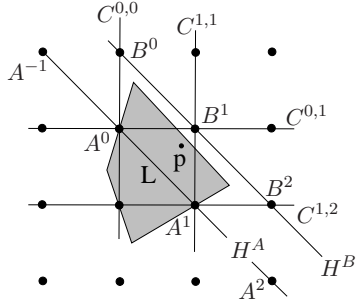}}
\caption{Illustration for the proof of Lemma~\ref{LEM:swiping} with
$m = 2$, $\ell = 0$ and $t = 1$.
Polytope $Q$ is shaded.}
\label{FIG:swipingpf}
\end{figure}

For all $i, j \in \mathbb{Z}$, define $C^{i,j}$ as the hyperplane containing
$A^i \cup B^j$.
%One can assume without loss of generality that
%$C^{j,j}$ is between $C^{i,i}$ and $C^{k,k}$ if and only if $i < j < k$.
For all $k \in \mathbb{Z}$, we have that $C^{0,k}$ and $C^{1, k+1}$
are the boundary hyperplanes of a split $(\pi^{k}, \pi^k_0)$.
If $Q \subseteq \{ x \in \mathbb{R}^m\ | \ \pi x \leq \pi_0 \}$,
then the lemma holds. So we assume that
$\inte \left( \overline{Q(\pi, \pi_0)}\right) \not= \emptyset$.
Let $\ell$ be the largest index $k$
such that $C^{0, k} \cap \inte \left( \overline{Q(\pi, \pi_0)}\right) =
\emptyset$. Let $t$ be the finite index such that
$p \in \conv(A^0 \cup B^{t}) \cup \inte \left(\conv(A^0 \cup
B^{t} \cup B^{t+1})\right)$. Consider the sequence of
${\PP}$-intersecting splits $(\pi^\ell$, $\pi^\ell_0)$, $\ldots$,
$(\pi^t$, $\pi^t_0)$. Let  $Q^k$ for $k = \ell, \ldots, t$ be the
polytopes obtained by applying this sequence of splits to $Q$. We
claim that $p \not\in Q^t$ and that, in fact, $p$ and
$\overline{Q^{t}(\pi, \pi_0)}$ are on opposite sides of the
hyperplane $C^{0,t+1}$. If the claim is correct,  applying the same
reasoning to each facet of ${\PP}$ in succession yields the lemma.

Let $U^k := \overline{Q^{k}(\pi, \pi_0)}$ for $k = \ell, \ldots, t$ and
let $\bar H^B$ be the half-space bounded by $H^B$ containing $H^A$.
We prove the claim by induction on $k$
by showing that $Q^{k}$ is contained on one side of $C^{0,k+1}$ and that
$C^{0,k+1} \cap U^k = F$ for $k = \ell, \ldots, t$.

For $k = \ell$, observe that $C^{0,\ell} \cap \overline{Q(\pi, \pi_0)} = F$
and thus
$C^{0, \ell} \cap Q$ is contained on one side of $C^{0,\ell+1}$.
Observe also that $C^{1, \ell+1} \cap Q$ is contained in the interior
of $\bar H^B$
and thus it is on the same side of $C^{0,\ell+1}$ as $C^{0, \ell} \cap Q$. It follows that
$Q^{\ell}$ is contained on one side of $C^{0,\ell+1}$ and
that the only points of $U^{\ell}$ on $C^{0,\ell+1}$ are points in $F$.
This proves the claim for $k = \ell$.

Suppose now that $k > \ell$ and that the induction hypothesis is true for
$k-1$. Observe that $C^{0, k} \cap U^{k-1} = F$ and that
$C^{1, k+1} \cap Q^{k-1}$ is contained  in the interior of $\bar H^B$
and thus is on the same side of $C^{0,k+1}$ as $C^{0, k} \cap Q^{k-1}$. It follows that
$Q^{k}$ is contained on one side of $C^{0,k+1}$ and
that the only points of $U^k$ on $C^{0,k+1}$ are points in $F$.
This proves the claim for $k > \ell$.
\end{proof}

\subsection{Proof of the main theorem.}
\label{SEC:MAIN}

Consider a full-dimensional polytope ${\PP} \subset \mathbb{R}^m$.
By a theorem of Chv\'atal~\cite{CCPS}, there exists a
finite sequence of Chv\'atal splits $S := ((\pi^1, \pi^1_0)$, $\ldots$,
$(\pi^c, \pi^c_0))$ such that, when
applying the sequence $S$ to ${\PP}$, the last polytope ${\PP}^c$ is the convex
hull ${\PP}_I$ of the integer points in ${\PP}$. If ${\PP} = {\PP}_I$, we define $c :=0$.
Otherwise, if ${\PP} \not= {\PP}_I$, define  ${\PP}^0 := {\PP}$ and,
for $j = 1, \ldots, c$, let ${\PP}^j$ be the sequence of polytopes
obtained from $S$. We may assume that ${\PP}^{j} \not= {\PP}^{j-1}$ for $j =
1, \ldots , c$. If ${\PP}_I$ is not full-dimensional, we set $t$ to be
the smallest index in $\{0, 1, \ldots, c-1\}$ such that $(\pi^{t+1},
\pi^{t+1}_0)$ is ${\PP}^t$-englobing. If no such englobing split exists 
in the sequence, we set $t := c$. It
follows that ${\PP}^j$ is full-dimensional for $j=0, \ldots, t$.
We say that ${\PP}$ has {\em Chv\'atal-index $t$} if $t$ is the smallest
possible value over all possible sequences of Chv\'atal splits as described
above.

\begin{lemma}
\label{OBS:englobing-split}
If ${\PP} \subset \mathbb{R}^m$ is a full-dimensional polytope with the $2$-hyperplane property and
Chv\'atal-index $t$, then there exists a sequence $S$ of $t$
Chv\'atal splits together with a split $(\pi^{t+1}, \pi^{t+1}_0)$ that is
${\PP}^t$-englobing.
\end{lemma}

\begin{proof}
If ${\PP}_I$ is full-dimensional then ${\PP}_I = {\PP}^t$ and ${\PP}_I$ is not
contained in a facet of ${\PP}$. Since ${\PP}$ has the 2-hyperplane
property, ${\PP}_I$ is 2-partitionable and therefore 
an ${\PP}_I$-englobing split exists. This is the split $(\pi^{t+1}, \pi^{t+1}_0)$.

If ${\PP}_I$ is not full-dimensional, we must have $c
> t$ and the definition of $t$ shows that $(\pi^{t+1},
\pi_0^{t+1})$ is ${\PP}^t$-englobing.
\end{proof}

We now prove  a result that implies Theorem~\ref{THM:main}.

\begin{theorem}
\label{THM:SuffGEN}
Let ${\PP}$ be a full-dimensional rational
lattice-free polytope in $\mathbb{R}^m$ such that ${\PP}$ has the
2-hyperplane property, and let $M_0 < M$ be finite numbers. For any
polyhedron $Q \subset R({\PP}, M, M_0)$, there exists a finite number
$q$ such that the height of the rank-$q$ split closure of $Q$ is at most
$M_0$.
\end{theorem}

\begin{proof}
If there exists an ${\PP}$-englobing split $(\pi, \pi_0)$, then for any 
polyhedron
$Q \subset R({\PP}, M, M_0)$, the height of $Q(\pi, \pi_0)$ is at most $M_0$ and
the theorem holds.

We prove the theorem
by induction on the dimension $m$ of ${\PP}$. If $m = 1$, ${\PP}$ cannot contain
more than two integer points and thus an ${\PP}$-englobing split
 exists and the result holds.

Assume now that ${\PP}$ has dimension $m \ge 2$ and that the result
holds for polytopes with strictly smaller dimension. Let $t$ be the
Chv\'atal-index of ${\PP}$. We make a second induction on $t$. More
precisely, the induction hypothesis is that the theorem holds for
any full-dimensional rational lattice-free polytope $K \subset
\mathbb{R}^k$ with either $k < m$, or $k=m$ and Chv\'atal-index $t'
< t$, for any finite numbers $M^*_0 < M^*$, and for any $Q^* \subset
R(K, M^*, M^*_0)$.

Assume first that $t=0$. Lemma~\ref{OBS:englobing-split} shows that
there exists an ${\PP}$-englobing split and therefore
the theorem holds. Assume now that $t > 0$. We consider the first 
Chv\'atal split $(\pi^1, \pi_0^1)$ in a sequence 
$(\pi^i, \pi_0^i)$, $i = 1, \ldots , t$ leading from ${\PP}$ to ${\PP}_I$.
The following claim, if true, shows that the effect of applying
$(\pi^1, \pi^1_0)$ on ${\PP}$ can be obtained by applying a finite sequence
of intersecting splits for polytopes satisfying the statement of
Lemma~\ref{LEM:Dj}.

{\sc Claim 1.}
%\label{claim:covering}
Let $(\pi^1, \pi_0^1)$ be a
Chv\'atal split for ${\PP}$ such that ${\PP}(\pi^1, \pi^1_0) \not= {\PP}$.
Then there exist a finite number $n \geq 1$, polytopes
${\PP}^i$, and splits $(\mu^i, \mu^i_0)$ such that $(\mu^i, \mu^i_0)$ is
${\PP}^i$-intersecting for $i = 1, \ldots, n$, and the following
properties hold, where $D^i = \cl({\PP}^i \setminus {\PP}^i(\mu^i, \mu^i_0))$:

\begin{itemize}
\item[(i)] ${\PP} \subseteq {\PP}^1$;
\item[(ii)] ${\PP} \setminus (D^1 \cup \ldots \cup D^i) \subseteq
{\PP}^{i+1}$ for $i = 1, \ldots , n-1$;
\item[(iii)] $\cl({\PP}\setminus {\PP}(\pi^1, \pi^1_0)) \subseteq
D^1 \cup \ldots \cup D^{n}$.
\end{itemize}

Before proving the claim, we show that it implies the theorem. We
use a third level of induction, this time on the number $n$ from
the above claim.
%If the claim holds for $t_1 = 0$, then $D^1 \cup
%\ldots \cup D^{t_1} = \emptyset$. Then condition (iii) from the
%claim implies that $\cl(P\setminus P(\pi^1, \pi^1_0)) = \emptyset$
%and so $P$ has Chv\'atal index at most $t-1$ and the result holds by
%the induction hypothesis stated above. We now assume that the claim holds
%for $t_1 > 0$ and
We assume that the theorem holds for any full-dimensional rational lattice-free
polytope $K \subset \mathbb{R}^k$, for any finite numbers $M^*_0 <
M^*$, for any $Q^* \subset R(K, M^*, M^*_0)$ with either
\begin{enumerate}
\item $k < m$, or
\item $k=m$ and $K$ has Chv\'atal-index $t' < t$, or
\item $k=m$, $K$ has Chv\'atal-index $t$ and there is a sequence of
length $n' < n$ satisfying Claim~1 for the first of the
$t$ Chv\'atal splits.
\end{enumerate}

As above, we assume that ${\PP}$ has Chv\'atal index $t > 0$ and
$Q \subset R({\PP}, M, M_0)$. Let $(\pi^1, \pi_0^1)$ be a Chv\'atal split such
that ${\PP}(\pi^1, \pi^1_0)$ has Chv\'atal index $t-1$.
Consider the sets defined in Claim~1.  By
Lemma~\ref{LEM:Dj} for ${\PP}$, ${\PP}^i$, $(\pi^i, \pi^i_0) := (\mu^i,
\mu^i_0)$ for $i = 1, \ldots, n$, $M^* := M$, $M^*_0 := M_0$ and $Q$,
we obtain $\Delta > 0$
and a finite sequence of splits $S$ such that after applying the
sequence $S$ to $Q$, one of the following cases holds:

\begin{itemize}
\item[(i)] the height of all points in $D^1 = \cl({\PP}^1 \setminus {\PP}^1(\mu^1,
\mu^1_0))$ is at most $M_0$ with respect to $Q(S)$.

\item[(ii)] the height of all points in $D^1 \cup \ldots \cup D^{n}$ is at
most $H - \Delta \cdot (M-M_0) $ with respect to $Q(S)$, where $H$ denotes the height of $Q$;
\end{itemize}

{\em Case (i):} We use the induction hypothesis on
$K :=\cl({\PP} \setminus D^1)$. Note that, by Lemma~\ref{LEM:QinR},
$Q(S) \subseteq R(K, H', M_0)$, where
$H'$ is the height of $Q(S)$, since the height of all points in
$D^1$ is at most $M_0$ with respect to $Q(S)$, and this is also the case
for all points not in the interior of ${\PP}$.

If $n=1$, then $K \subseteq {\PP}(\pi^1,\pi^1_0)$ by (iii) of
Claim~1. Since
${\PP}(\pi^1,\pi^1_0)$ has Chv\'atal index $t-1$, it
satisfies the second induction hypothesis and the result holds.

If $n \geq 2$, observe that ${\PP}^i$ and $(\mu^i, \mu^i_0)$ for $i=2,
\ldots, n$ show that the claim is satisfied with a sequence of
$n -1$ splits for $K$. By definition of the Chv\'atal index, $K$
is full-dimensional and therefore $K$ satisfies the third induction hypothesis.

Therefore, by induction, there exists a finite $q'$ such that the height of the
rank-$q'$ split closure $Q^1$ of $Q(S)$ is at most $M_0$. Since the split
closure of a polyhedron is a polyhedron~(Cook et al.~\cite{COOK}), this implies
the existence of a finite sequence $S'$ of splits that gives $Q^1$
when applied to $Q(S)$. Applying the sequence $S$ followed by the
sequence $S'$ on the polyhedron $Q$, we reduce its height to at most
$M_0$, proving the theorem.

\bigskip
{\em Case (ii):} If $M_0 \ge H - \Delta \cdot (M-M_0)$, then in particular
the height of all points in $D^1$ is at most $M_0$ with respect to $Q(S)$ and
Case (i) applies. Therefore we may assume that $M_0 \le H - \Delta \cdot (M-M_0)$.
Case (ii) combined with conclusion (iii) in
Claim~1 that $\cl({\PP}\setminus {\PP}(\pi^1, \pi^1_0))
\subseteq D^1 \cup \ldots \cup D^{n}$, implies that the
height of all points in $\cl({\PP}\setminus {\PP}(\pi^1,\pi^1_0))$ is at
most $H - \Delta \cdot (M-M_0)$ with respect to $Q(S)$.
Recall that all points not in the interior of ${\PP}$ have height at most
$M_0 \le H - \Delta \cdot (M-M_0)$ with respect to $Q$ and thus also with respect to
$Q(S)$. Let $H(S)$ be the height of $Q(S)$. Then, by Lemma~\ref{LEM:QinR},
$Q(S) \subseteq R({\PP}(\pi^1,\pi^1_0), H(S), H-\Delta \cdot (M-M_0))$. Since
the Chv\'atal-index of ${\PP}(\pi^1,\pi^1_0)$ is $t-1$, we can apply
the second induction hypothesis to $Q(S)$. Therefore, there exists
a finite sequence $S^1$ of splits that gives a polyhedron $Q^1$ of height at most $H-\Delta \cdot (M-M_0)$
when applied to $Q(S)$. We can get $Q^1$ from $Q$ by applying the sequence $S$ followed by $S^1$.
Note that $Q^1 \subseteq Q \subseteq R({\PP}, M, M_0)$ so Lemma~\ref{LEM:Dj} can be applied to $Q^1$
with the same $\Delta$ and sequence of splits $S$, and Case (i) or (ii) will hold,
where the only change in Case (ii) is the height of $Q^1$, which is at most $H-\Delta \cdot (M-M_0)$.
Therefore let us apply the sequence $S$ on $Q^1$. If we end up in Case (i) then we are done by the arguments in
Case (i). Otherwise we have the situation of Case (ii) where the height of all points in $D^1 \cup \ldots \cup D^{n}$ is at most $H - 2\Delta \cdot (M-M_0) $ with respect to $Q^1(S)$.
If $M_0 \ge H - 2\Delta \cdot (M-M_0)$, the theorem follows as in Case (i).
If $M_0 \le H - 2\Delta \cdot (M-M_0)$, we can apply the second induction hypothesis to $Q^1(S)$.
Therefore, there exists a finite sequence $S^2$ of splits that we can apply to $Q^1(S)$ to
obtain $Q^2$ with height at most $H - 2\Delta \cdot (M-M_0)$.
We can get $Q^2$ from $Q$ by applying the sequence $S$ followed by $S^1$ followed by $S$ followed by $S^2$.
Continuing this process at most $\lceil \frac{H - M_0}{\Delta \cdot (M-M_0)}
\rceil$ times, we end up in Case (i) at some point, proving the
theorem.

\bigskip
It thus remains to prove the claim.

\begin{proof}[Proof of Claim~1]
Let $H^A$ and $H^B$ be the boundary hyperplanes of the Chv\'atal split
$(\pi^1, \pi_0^1)$ such that ${\PP}^* := H^A \cap {\PP} \neq \emptyset$.
By assumption, ($\pi^1,\pi^1_0$) is not ${\PP}$-englobing and
${\PP}(\pi^1, \pi^1_0) \not= {\PP}$. Therefore, ${\PP}^*$ is not a face of ${\PP}$ and thus
$\dim({\PP}^*) = m-1$ and ${\PP}^*$ is lattice-free in its affine space.
Moreover, both $H^A$ and $H^B$ are lattice subspaces.

Let $W$ be obtained using Theorem~\ref{LEM:enlarge} on ${\PP}^*$,
and let ${\PP}_s := \conv({\PP} \cup W)$. Let $d$ be a vector joining
an integer point in $H^A$ to an integer point in $H^B$. Let $q \in
\relint(W)$ and let $p$ be a point on the boundary of ${\PP}_s$
between $H^A$ and $H^B$ such that the line joining $q$ and $p$ is in 
the direction $d$ (Figure~\ref{FIG:unimod}, left). Apply
Lemma~\ref{LEM:swiping} to $Q := {\PP}_s$, ${\PP} := W$,  $p$, and $(\pi^1, \pi_0^1)$. This
yields a sequence $(\mu^1, \mu^1_0), \ldots, (\mu^r, \mu^r_0)$ of
$W$-intersecting splits with the following property.
Let ${\PP}^1 = {\PP}_s$ and, for $i = 1, \ldots, r$, define ${\PP}^{i+1}$ as the polytope obtained from ${\PP}^i$ by
applying the split $(\mu^i, \mu^i_0)$.  Then Lemma~\ref{LEM:swiping} shows that ${\PP}^{r+1} \cap \{x \ | \
\pi^1x \geq \pi^1_0\}$ is contained in $\conv (p \cup W)$.

\begin{figure}[ht]
\centering \scalebox{0.8}{
\includegraphics{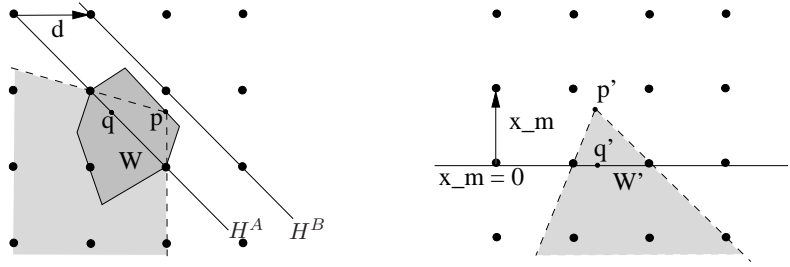}}
\caption{Illustration for the proof of Theorem~\ref{THM:SuffGEN}. On the left,
polytope ${\PP}_s \subseteq \mathbb{R}^2$ is shaded, $W$ is a segment, and part of
cone $U$ is lightly shaded with its boundary in dashed lines. On the right, 
the images of $W$, $U$, $p$, and $q$ after
applying the unimodular transformation $\tau$ are depicted.}
\label{FIG:unimod}
\end{figure}

Let $U$ be the cone with apex $p$ and rays joining $p$ to the
vertices of $W$.  Observe that $U \supseteq {\PP}^{r+1}$.
Let $\mathcal{T}: \mathbb{R}^{m} \rightarrow \mathbb{R}^{m}$
be a unimodular transformation mapping the vector $d$ to
$(0, \ldots, 0, 1) \in \mathbb{R}^m$ and mapping $H^A$ to
$\mathbb{R}^{m} \cap \{x \ | \ x_m = 0\}$ (Figure~\ref{FIG:unimod}, right). 
Note that
$\mathcal{T}(\mathbb{Z}^m \cap H^A) = \mathbb{Z}^{m} \cap \{x \ | \ x_m = 0\}$
by Corollary 4.3a in Schrijver~\cite{sch}.

Let $p' :=  \mathcal{T}(p)$, $U' := \mathcal{T}(U)$, and
$W' := \mathcal{T}(W)$.
As ${\PP}$ has the 2-hyperplane property, and as
${\PP}^*$ is the intersection of ${\PP}$ with one boundary hyperplane of a Chv\'atal
split, the convex hull $G$ of the integer points in ${\PP}^*$ is a face of ${\PP}_I$.
Moreover, if a face of $G$ is contained in a facet of ${\PP}$ then it is contained
in a facet of ${\PP}^*$. It follows that ${\PP}^*$ has the 2-hyperplane property
and by applying Theorem~\ref{LEM:enlarge} to ${\PP}^*$ we obtain that
so does $W$. As $\mathcal{T}$ is unimodular, the transformation
$\mathcal{T}$ maps splits in $H^A$ to splits in
$\mathbb{R}^{m} \cap \{x \ | \ x_m = 0\}$. Therefore, $W'$ also has the
2-hyperplane property. We will now apply the first induction
hypothesis to $K := W' \subseteq \mathbb{R}^{m-1}$,
$Q^* := U'\subseteq \mathbb{R}^m$, $M^*$ equals the height of $U'$ and
$M^*_0 = 0$. Here, the crucial point to follow this induction step
is to understand that variable $x_m$ plays the role of variable $z$.
Thus, all splits obtained from this induction are parallel to the $x_m$-axis.
Note that, by Lemma~\ref{LEM:QinR},
$Q^* \subseteq R(K, M^*, M^*_0)$. By induction, there exists a finite
sequence of splits, all parallel to the $x_m$-axis, whose application on
$U'$ reduces its height to $0$, or equivalently, removes $\conv(p' \cup W')$.
Using the inverse of $\mathcal{T}$, this yields a
sequence $(\mu^{r+1}, \mu^{r+1}_0), \ldots, (\mu^{n},
\mu^{n}_0)$ of splits whose boundary hyperplanes are all parallel to $d$
and such that its application to $U$ removes $\conv(p \cup W)$. 
Since $-d$ is in the interior of the full-dimensional
recession cone of $U$, the sequence $(\mu^{r+1}, \mu^{r+1}_0),
\ldots, (\mu^{n}, \mu^{n}_0)$ is $U$-intersecting. One last hurdle remains,
as $U$ is unbounded and we need to construct polytopes ${\PP}^i$ such that
$(\mu^i, \mu^i_0)$ is ${\PP}^i$-intersecting. To this end, let $U^{r+1}, \ldots, U^n$ 
be the polyhedra obtained by applying the sequence 
$(\mu^{r+1}, \mu^{r+1}_0), \ldots, (\mu^{n}, \mu^{n}_0)$ to $U$. For $j = r+1,
\ldots, n$, let $w^{1, j}$ and $w^{2, j}$ be two integer points in
$U^j$ contained in each of the two boundary
hyperplanes of $(\mu^j, \mu_0^j)$ respectively. It is possible to truncate $U$
into a polytope $U^*$ by intersecting $U$ with a half-space bounded
by a hyperplane parallel to $H^A$, such that ${\PP}^{r+1}$ and all the points
$w^{1,j}, w^{2,j}$ are contained in $U^*$. As all these points are
integer, they are contained in the polytope obtained
by applying the sequence $(\mu^{r+1}, \mu^{r+1}_0), \ldots,
(\mu^{n}, \mu^{n}_0)$ of splits to $U^*$, proving that the
sequence is a $U^*$-intersecting sequence.

Redefine ${\PP}^{r+1}$ to be equal to $U^*$ and note that this redefinition is an enlargement. For  $i = r+1, \ldots, n$, define ${\PP}^{i+1}$ as
the polytope obtained from ${\PP}^i$ by applying the split $(\mu^{i},
\mu^{i}_0)$. The
sequence $(\mu^i, \mu^i_0)$ and polytopes ${\PP}^i$ for $i = 1, \ldots,
n$ are those used to show that Claim~1 is
satisfied. Indeed, by construction, ${\PP} \subseteq {\PP}_s = {\PP}^1$, therefore
point (i) holds. Observe that since ${\PP} \subseteq {\PP}^1$,
${\PP} \setminus (D^1 \cup \ldots \cup D^i) \subseteq {\PP}^1
\setminus (D^1 \cup \ldots \cup D^i) = {\PP}^{i+1}$ for $i = 1, \ldots, r-1$.
Moreover, because the redefinition of ${\PP}^{r+1}$ is an enlargement,
${\PP}^r \setminus D^r \subseteq {\PP}^{r+1}$. Consequently,
${\PP}^1 \setminus (D^1 \cup \ldots \cup D^i) \subseteq {\PP}^{i+1}$ for
$i = 1, \ldots, r$. Also note that ${\PP}^i \setminus D^i = {\PP}^{i+1}$ for
$i = r+1, \ldots, n$. Therefore, point (ii) holds. Finally, point (iii)
holds since ${\PP}^1\setminus (D^1 \cup \ldots \cup D^r)$ is contained in
$\conv(p \cup W)$ and $D^{r+1} \cup \ldots \cup D^n$ contains
$\conv(p \cup W)$. Therefore $D^1 \cup \ldots \cup D^n$ contains
$\cl({\PP}\setminus {\PP}(\pi^1, \pi^1_0))$.
\end{proof}

This completes the proof of Theorem~\ref{THM:SuffGEN}  \end{proof}

We can now prove Theorem~\ref{THM:main}, completing
the proof of Theorem~\ref{THM:GenDim}.

\begin{proof}[Proof of Theorem~\ref{THM:main}]
Since $f \in \inte({\PP})$ and ${\PP}$ satisfies the assumptions of
Theorem~\ref{THM:SuffGEN}, we have that ${\QQ}^{\PP} \subset R({\PP}, 1, 0)$.
Theorem~\ref{THM:SuffGEN} implies that there
exists a finite number $q$
such that the height of the rank-$q$ split closure of ${\QQ}^{\PP}$ 
is at most zero.
By Observation~\ref{OBS:rankQL}, this implies that the split rank with 
respect to $\QQ$ of the ${\PP}$-cut is at most $q$.
\end{proof}

\section*{Acknowledgments} We would like to thank Christian Wagner
and two anonymous referees for their helpful comments on the first draft
of this paper. %This research was supported by ONR grant N00014-09-1-0033, ANR grant ANR06-BLAN-0375, NSF grant CMMI1024554, and a Mellon Fellowship.}

\end{document}